\documentclass[11pt]{article}
\usepackage{hyperref}
\usepackage{enumerate, amsthm}
\usepackage{rotating}
\usepackage{cancel}
\usepackage{xcolor}
\usepackage[margin={1.5cm, 1.5cm},font=small, labelsep=endash]{caption} 
\usepackage[margin=1 in]{geometry}
\usepackage{mathtools}

      \usepackage{amssymb}
      \usepackage{amsmath}

      \numberwithin{equation}{section}
      \theoremstyle{plain}
      \newtheorem{theorem}{Theorem}[section]
      \newtheorem{lemma}[theorem]{Lemma}

      \theoremstyle{definition}

      \theoremstyle{remark}
      \newtheorem{remark}[theorem]{Remark}

      \newcommand{\E}{\mathbb E}

      \makeatletter
      \def\@setcopyright{}
      \def\serieslogo@{}
      \makeatother

\title{Percolation of averages in the stochastic mean field model:\\ the near-supercritical regime}
\author{ Jian Ding\thanks{Partially supported by NSF grant DMS-1313596.} \\ University of Chicago \and Subhajit Goswami\footnotemark[1]  \\
University of Chicago
}

\begin{document}

 \maketitle

 \begin{abstract}
 For a complete graph of size $n$, assign each edge an i.i.d.\ exponential variable with mean $n$. 
 For $\lambda>0$, consider the length of the longest path whose average weight is at most 
 $\lambda$. It was shown by Aldous (1998) that the length is of order $\log n$ for $\lambda < 
 1/\mathrm{e}$ and of order $n$ for $\lambda > 1/\mathrm{e}$. In this paper, we study the near-
 supercritical regime where $\lambda = \mathrm{e}^{-1} +\eta$ with $\eta>0$ a small fixed number. 
 We show that there exist two absolute constants $C_1, C_2 > 0$ such that with high probability 
 the length is in between $n \mathrm{e}^{-C_1/\sqrt{\eta}}$ and $n \mathrm{e}^{-C_2/\sqrt{\eta}}$. 
 Our result corrects a non-rigorous prediction of Aldous (2005).\newline 
 
 \end{abstract}

\section{Introduction}
Let $\mathcal W_n$ be a complete undirected graph of $n$ vertices where each edge is assigned an 
independent exponential weight with mean $n$; this is referred to as the \emph{stochastic mean-
field} ($\mbox{SMF}_n$) model. For a (self-avoiding) \emph{path} $\pi = (v_0, v_1,\ldots, v_m)$,  
define its length $len(\pi)$ and average weight $A(\pi)$ by
\begin{align*}
 len(\pi) = m \,, \mbox{ and }A(\pi) = \tfrac{1}{m}\mbox{$\sum_{i = 1}^m$} W_{(v_{i-1},v_i)}\,,
\end{align*}
where $W_{(u, v)}$ is the weight of the edge $(u, v)$. For $\lambda >0$, let $L(n,\lambda)$  be 
the length of the longest path with average weight below $\lambda$, i.e., 

$$L(n,\lambda) = \max\{len(\pi): A(\pi) \leq \lambda, \mbox{ }\pi\mbox{ is a path in 
$\mbox{SMF}_n$ model}\}\,.$$

In a non-rigorous paper of Aldous \cite{Aldous05}, it was predicted that $L(n, \lambda) \asymp n 
(\lambda - \mathrm{e}^{-1})^\beta$ with $\beta = 3$ $\lambda \downarrow \mathrm{e}^{-1}$. Our main 
result is the following theorem, which corrects Aldous' prediction.
\begin{theorem}
\label{Prop}
Let $\lambda = 1/\mathrm{e} + \eta$ where $\eta > 0$. Then there exist absolute constants $C_1, 
C_2 , \eta^*>0$ such that for all $\eta \leq \eta^*$,
\begin{equation}
\label{main_prop}
\lim_{n\to \infty}\mathbb{P}\big(n \mathrm{e}^{- C_1 / \sqrt{\eta}} \leq L(n,\lambda)\leq n 
\mathrm{e}^{- C_2 / \sqrt{\eta}}\big) = 1\,.
\end{equation}
\end{theorem}

The study of the object $L(n, \lambda)$ was initiated by Aldous \cite{Aldous98} where a phase 
transition was discovered at the threshold $\mathrm{e}^{-1}$. It was shown that with high 
probability $L(n, \lambda)$ is of order $\log n$ for $\lambda < \mathrm{e}^{-1}$ and $L(n, 
\lambda)$ is of order $n$ when $\lambda > \mathrm{e}^{-1}$. The critical behavior was established  
in \cite{Ding13}, where it was proved that with high probability $L(n, \lambda)$ is of order 
$(\log n)^3$ when $\lambda$ is around $\mathrm{e}^{-1}$ within a window of order $(\log n)^{-2}$. 
Our Theorem~\ref{Prop} describes the behavior in the near-supercritical regime, and in 
particular states that $L(n, \lambda)/n$ is a stretched exponential in $\eta$ with $\eta = \lambda 
- \mathrm{e}^{-1} \downarrow 0$. Another interesting result proved in \cite{Ding13} states that 
$L(n, \lambda) \geq n^{1/4}$ in a somewhat similar regime namely $\lambda \geq 1 / \mathrm{e} 
+ \beta(\log n)^{-2}$, where $\beta > 0$ is an absolute constant. Notice that substituting $\eta = 
C(\log n)^{-2}$ in \eqref{main_prop}, we indeed get a fractional power of $n$. In fact our method 
should work, subject to some technical modifications, all the way down to $\eta = C(\log 
n)^{-2}$ for a large absolute constant $C$. However, we do not attempt any rigorous proof of 
this in the current paper. 

A highly related question is the length for the cycle of minimal mean weight, which was studied by 
Mathieu and Wilson \cite{MW13}. An interesting phase transition was found in \cite{MW13} with 
critical threshold $\mathrm{e}^{-1}$ on the mean weight. Further results on this problem have been 
proved in \cite{DSW}. It might be relevant to mention here that the method used in \cite{DSW} 
could be potentially useful for nailing down the second phase transition detected in 
\cite{Ding13}, namely the transition from $\eta = \alpha (\log n)^{-2}$ to $\eta = \beta (\log 
n)^{-2}$ where $\alpha, \beta$ are positive constants.

Another related question is the classical travelling salesman problem (TSP), where one minimizes 
the weight of the path subject to passing through every single vertex in the graph. For the TSP in 
the mean-field set up, W\"{a}stlund \cite{Wastlund10} established the sharp asymptotics for more 
general distributions on the edge weight, confirming the Krauth-M\'{e}zard-Parisi conjecture
\cite{MP86b, MP86, KM89}. Indeed, it is an interesting challenge to give a sharp estimate on $L(n, 
\lambda)$ for $\mathrm{e}^{-1} <\lambda < \lambda^*$ (here $\lambda^*$ is the asymptotic value for 
TSP), interpolating the critical behavior and the extremal case of TSP. A question of the same 
flavor on steiner tree is given in \cite{Bollobas04}.

One can also look at the maximum size of tree with average weight below a certain threshold, 
where a phase transition was proved in \cite{Aldous98}. The extremal case of the question on the 
tree with minimal average weight is the well-known minimal spanning tree problem, where a 
$\zeta(3)$ limit is established by Frieze \cite{Frieze85}.

\noindent {\bf Main ideas of our proofs.}  A straightforward first moment computation as done in 
\cite{Aldous98} implies that $\lim_{n \to \infty}\mathbb{P}\big(L(n, \lambda) = O(\log n)\big) = 
1$ when $\lambda < 1/\mathrm{e}$ (see also \cite[Theorem 1.3]{Ding13}). For $\lambda> 
1/\mathrm{e}$, a sprinkling method was employed in \cite{Aldous98} to show that with high 
probability $L(n, \lambda) = \Theta(n)$. The author first proved that with high probability there 
exist a large number of paths with average weight slightly above $1/\mathrm{e}$ and then used a 
certain greedy algorithm to connect these paths into a single long path with average weight 
slightly above $1/\mathrm{e}$. However, the method in \cite{Aldous98} was not able to describe the 
behavior at criticality. In \cite{Ding13} (see also \cite{MW13} for the cycle with minimal average 
weight), a second moment computation was carried out restricted to paths of average weight below 
$1/\mathrm{e}$ and with the maximal deviation (defined in \eqref{eq-def-M} below) at most $O(\log 
n)$, thereby yielding that with high probability $L(n, 1/\mathrm{e}) = \Theta((\log n)^3)$. A 
crucial fact responsible for the success of the second moment computation is that the length of 
the target path is $\Theta((\log n)^3) \ll \sqrt{n}$. As such, a straightforward adaption of this 
method would not be able to succeed in the regime considered by this paper. 

TSP, where one studies paths (cycles) that visit every single vertex, is in a sense analogous to 
the question of finding the minimal value $\lambda$ for which $L(n, \lambda) = n$ with high 
probability. W\"{a}stlund \cite{Wastlund10} showed that the minimum average cost of TSP converges 
in probability to a positive constant by relaxing it to a certain linear optimization problem. But 
it seems difficult to extend his method to ``incomplete'' TSP i.e. when the target object is the 
minimum cost cycle having at least $pn$ many edges for some $p \in (0,1)$. Since our problem is in 
a sense dual to incomplete TSP in the regime we are interested in, the method of \cite{Wastlund10} 
does not seem to be suitable for our purpose either. In the current work, our method is inspired 
by the (first and) second moment method from \cite{Ding13, MW13} as well as the sprinkling method 
employed in \cite{Aldous98}. 

In order to prove the upper bound, our main intuition is that if $L(n, \lambda)$ were greater than 
$\mathrm{e}^{-C_2/\sqrt{\eta}} n$ then we would have a larger number of short and light paths (a 
light path refers to a path with small average weight --- at most a little above $1/\mathrm{e}$) 
than we would typically expect. Formally, let $\ell = \frac{c_1}{\eta}$ where $c_1$ is a small 
positive constant, and consider the number of paths (denoted by $N_{\eta / c_1, c_2}$) with length 
$\ell$ and total weight no more than $\lambda \ell - c_2 \sqrt{\ell}$ for a positive constant 
$c_2$. We call such a path a \emph{downcrossing}. A straightforward computation gives $\mathbb E 
N_{\eta / c_1, c_2} = O(1) n \ell \mathrm{e}^{-c_3/\sqrt{\eta}}$ for a positive constant $c_3$ 
depending on $c_1$ and $c_2$. Now we consider the number of paths (denoted by $N_\delta$) of 
length $\delta(\lambda) n$ and average weight at most $\lambda$. Such paths have two 
possibilities: (1) The path contains substantially more than $\mathbb E N_{\eta / c_1, c_2}$ many 
downcrossings, which is unlikely by Markov's inequality. (2) The path does not have substantially 
more than $\mathbb E N_{\eta / c_1, c_2}$ many downcrossings. This is also unlikely for the 
following reasons: (a)  A straightforward first moment computation gives that $\mathbb E N_\delta  
= O(n) \mathrm{e}^{c_4 \delta n \eta}$ for a constant $c_4 > 0$; (b) The number of downcrossings 
along a path of this kind, or a random variable that is ``very likely'' smaller, should dominate a 
Binomial random variable $\mathrm{Bin}(\delta n/\ell, c_5)$ where $c_5 > 0$ is an absolute 
constant (since in the random walk bridge, every subpath of size $\ell$ has a positive chance to 
have such a downcrossing). If we choose $\delta$ suitably large as in Theorem~\ref{Prop}, we are 
suffering a probability cost for the constraint on the number of downcrossings (probability for a 
binomial much smaller than its mean) and this probability cost is of magnitude $\mathrm{e}^{-c_6 
\delta n /\ell}$ for a constant $c_6 > 0$ depending in $c_1$. If we choose $c_1$ small enough this 
probability cost kills the growth of $\mathrm{e}^{c_4 \delta n \eta}$ in $\mathbb E N_\delta$. 
Therefore, paths of this kind do not exist either. The details are carried out in 
Section~\ref{sec-upper}.

For the lower bound, our proof consists of two steps. In light of the preceding discussion, we 
cannot hope to directly apply a second moment method from \cite{Ding13, MW13} to show the 
existence of a light path that is of length linear in $n$. As such, in the first step of our proof 
we prove that with high probability there exists a linear (in $n$) number of disjoint paths, each 
of which has weight slightly below $\lambda$ and is of length $\mathrm{e}^{c_7/\sqrt{\eta}}$ for 
an absolute constant $c_7>0$. This is achieved by two second moment computations, which are 
expected to succeed as the length of the path under consideration is $\ll \sqrt{n}$ (indeed 
remains bounded as $n\to \infty$). In the second step, we propose an algorithm which, with 
probability going to 1, strings together a suitable collection of these short light paths to form 
a light path of length $\mathrm{e}^{-c_8/\sqrt{\eta}}n$ for an absolute constant $c_8>0$. Our 
algorithm is similar to the greedy algorithm (or in a different name exploration process) employed 
in \cite{Aldous98}. But in order to ensure that the additional weight introduced by these 
connecting bridges only increases the average weight of the final path by at most a multiple of 
$\eta$, we have to use a more delicate algorithm. The details are carried out in Section~\ref{sec-lower}.

\noindent {\bf Notation convention.} 
 For a graph $G$, we denote by $V(G)$ and $E(G)$ the set of vertices and edges of $G$ 
 respectively. A path in a graph $G$ is an (finite) ordered tuple of vertices $(v_0, v_1, \cdots, 
 v_m)$, all distinct. For a path $\pi = (v_0, v_1, \cdots, v_m)$, we also use $\pi$ to denote the 
 graph whose vertices are $v_0, v_1, \cdots, v_m$ and edges are $(v_0, v_1), \cdots, (v_{m-1}, 
 v_m)$. This would be clear from the context. The weight of an edge $e$ in $\mathcal{W}_n$ is denoted by $W_e$ and we define the  total weight $W(\pi)$ of a path $\pi$ as $\sum_{e \in E(\pi)} W_{e}$. The collection of all paths in $\mathcal{W}_n$ of length $\ell \in [n]$ is denoted as $\Pi_\ell$. We let $\lambda  = 1/\mathrm{e} + \eta$ where $\eta$ is a  fixed positive number. A  path is called \emph{$\lambda$-light} if its average weight is at most  $\lambda$, and a path is  called  \emph{$(\lambda, C)$-light} if its total weight is at most $\lambda \ell - C\sqrt{\ell}$ where  $\ell$ is length of 
 the path. For nonnegative real or integer valued variables $x_0, x_1, \cdots, x_n$, let $S$ be a 
 statement involving $x_0, x_1, \cdots, x_n$. We say that $S$ holds ``for large  $x_0$ (given 
 $x_1, \cdots, x_n$)'' or ``when $x_0$ is large (given $x_1, \cdots, x_n$)'' if it  holds for any 
 fixed values of $x_1, \cdots, x_n$ in their respective domains and $x_0 \geq a_0$  where $a_0$ is 
 some positive number depending on the fixed values of $x_1, \cdots, x_n$. In case  $a_0$ is an 
 absolute constant, the phrase ``(given $x_1, \cdots, x_n$)'' will be dropped. We use  ``for small 
 $x_0$'' or ``when $x_0$ is small'' with or without the qualifying phrase ``(given  $x_1, x_2, 
 \cdots, x_n$)'' in similar situations if the statement $S$ holds instead for $0 < x_0  \leq a_0$. 
 Throughout this paper the order notations $O(.), \Theta(.), o(.)$ etc.~are assumed to be with 
 respect to $n \to \infty$ while keeping all the other involved parameters (such as $\ell$, $\eta$ 
 etc.) fixed. We will use $C_1, C_2, \ldots$ to denote constants, and each  $C_i$ will denote the 
 same number throughout of the rest of the paper.
\smallskip

\noindent {\bf Acknowledgements.} We are grateful to David Aldous for very useful discussions, and 
we thank an anonymous referee for a careful review of an earlier manuscript and suggesting a 
simpler proof of Lemma~\ref{count_vertices}.

\section{Proof of the upper bound}
\label{sec-upper}

Let $\eta'$ be a multiple of $\eta$ by a constant bigger than 1 whose precise value is to be 
selected. Set $\ell = \lfloor 1 / \eta' \rfloor$ and let $N_{\eta'}$ be the number of ``$(\lambda, 
1)$-light'' paths of length $\ell$. We assume $\eta < 1$ so that $\ell \geq 1$. As outlined in the 
introduction, we shall first control $N_{\eta'}$. 

It is clear that the distribution of the total weight of a path of length $k$ follows a Gamma 
distribution $\Gamma(k, 1/n)$, where the density $f_{\theta, k}(z)$ of $\mathrm{Gamma}(k, \theta)$ 
is given by
\begin{equation}
\label{gamma_density}
f_{\theta, k}(z) = \theta^k z^{k - 1}\mathrm{e}^{-\theta z} / (k - 1)!\mbox{ for all }z \geq 0, 
\theta > 0\mbox{ and }k\in \mathbb{N}.
\end{equation}
By \eqref{gamma_density} and the Stirling's formula, we carry out a straightforward computation 
and get that
\begin{eqnarray}
\mathbb{E} N_{\eta'} & = & (1 + o(1)) \times n^{\ell+1} \times \mathbb{P}\Big ( \mathrm{Gamma}
(\ell,1/n) \leq \lambda \ell - \sqrt{\ell} \Big ) \nonumber\\
			& = & (1 + o(1)) \times n^{\ell+1}\times \frac{\mathrm{e}^{-(\lambda \ell - 
			\sqrt{\ell})/n}(\lambda \ell - \sqrt{\ell})^{\ell}}{\ell!n^{\ell}}\nonumber\\
\label{first_bnd}& = & (1 + o(1))C_0(\eta)\alpha\mathrm{e}^{\mathrm{e}\eta/\eta'} \sqrt{\eta'} 
\mathrm{e}^{-\mathrm{e}/\sqrt{\eta'}}n,
\end{eqnarray}
where $C_0(\eta) \to 1$ as $\eta \to 0$, and $\alpha$ is a positive constant. Furthermore the 
factors $1 + o(1)$ are strictly less than 1.
\\
We also need a bound on the second moment of $N_{\eta'}$ to control its concentration around 
$\mathbb{E}N_{\eta'}$. For $\gamma \in \Pi_\ell$, define $F_{\gamma}$ to be the event that 
$\gamma$ is $(\lambda, 1)$-light. Then clearly we have $N_{\eta'} = \sum_{\gamma \in 
\Pi_\ell}\mathbf{1}_{F_{\gamma}}$. In order to compute $\E (N_{\eta'})^2$, we need to estimate 
$\mathbb{P}(F_{\gamma} \cap F_{\gamma'})$ for $\gamma, \gamma' \in \Pi_\ell$. In the case 
$E(\gamma) \cap E(\gamma') = \emptyset$, we have $F_\gamma$ and $F_{\gamma'}$ independent of each 
other and thus $\mathbb{P}(F_{\gamma'}|F_{\gamma}) = \mathbb{P}(F_{\gamma'})$. In the case $|
E(\gamma) \cap E(\gamma')| = j > 0$, we have
\begin{eqnarray}
\label{cond_prob_order}
\mathbb{P}(F_{\gamma'}|F_{\gamma}) &\leq & \mathbb{P}\big(\mathrm{Gamma}(\ell-j,1/n) \leq \lambda 
\ell\big)
                                       \leq \tfrac{1}{(\ell - j)!}\tfrac{(\lambda \ell)^{\ell - 
                                       j}}{n^{\ell - j}}.
\end{eqnarray}
Further notice that if $|E(\gamma) \cap E(\gamma')| = j$, then $|V(\gamma) \cap V(\gamma')|$ is at 
least $j + 1$ as $\gamma \cap \gamma'$ is acyclic. So given any $\gamma \in \Pi_\ell$, the number 
of paths $\gamma'$ such that $|E(\gamma) \cap E(\gamma')| = j$ is at most $O(n^{\ell - j})$.  
Altogether, we obtain that
\begin{align}
\mathbb{E} N_{\eta'}^2 & = \mbox{$\sum_{\gamma, \gamma' \in \Pi_\ell}$}\mathbb{P}(F_{\gamma}\cap 
F_{\gamma'})  = \mbox{$\sum_{\gamma \in \Pi_\ell}$}\mathbb{P}(F_{\gamma})\mbox{$\sum_{\gamma' \in 
\Pi_\ell}$}\mathbb{P}(F_{\gamma'}|F_{\gamma}) \nonumber \\
                          & \leq \mbox{$\sum_{\gamma \in \Pi_\ell}$}\mathbb{P}
                          (F_{\gamma})\Big(\sum_{\gamma':E(\gamma' \cap 
                          \gamma)=\emptyset}\mathbb{P}(F_{\gamma'}) + \sum_{1 \leq j \leq 
                          \ell} \sum_{\gamma':|E(\gamma' \cap \gamma)|=j}\frac{1}{(\ell - 
                          j)!}\frac{(\lambda \ell)^{\ell - j}}{n^{\ell - j}}\Big) \nonumber\\
                          & = \mbox{$\sum_{\gamma \in \Pi_\ell}$}\mathbb{P}
                          (F_{\gamma})\Big(\sum_{\gamma':E(\gamma' \cap 
                                                    \gamma)=\emptyset}\mathbb{P}(F_{\gamma'}) + 
                                                    \sum_{1 \leq j \leq 
                                                    \ell}\frac{O(n^{\ell - j})}{(\ell - 
                                                    j)!}\frac{(\lambda \ell)^{\ell - j}}{n^{\ell - 
                                                    j}}\Big) \nonumber\\
                          &\leq \mbox{$\sum_{\gamma \in \Pi_\ell}$}\mathbb{P}
                          (F_{\gamma})\Big(\mathbb{E}N_{\eta'} + O(1) \Big) = 
                          \mathbb{E}N_{\eta'}\Big(\mathbb{E}N_{\eta'} + O(1) \Big).  
                          \label{first_second_moment}
\end{align}
Since $\E N_{\eta'} = \Omega(1)$ as implied by \eqref{first_bnd}, \eqref{first_second_moment} 
yields that
\begin{equation}
\label{concentration_1}
\mathbb{E}N_{\eta'}^2 = (\mathbb{E}N_{\eta'})^2 (1 + o(1)).
\end{equation}
As a consequence of Markov's inequality (applied to $|N_{\eta'} - \E N_{\eta'} |^2$), we get that
\begin{equation}
\label{concentration_2}
\mathbb{P}\big(N_{\eta'} \geq 2 \mathbb{E}N_{\eta'}\big) = o(1).
\end{equation}

Next, we set out to show that any long $\lambda$-light path should have a large number of subpaths 
which are $(\lambda, 1)$-light. Let $\pi$ be a path of length $\delta n$ for some $\delta > 
0$. Denote its successive edge weights by $X_1, X_2, \ldots X_{\delta n}$ and let $S_k = 
\sum_{i=1}^k X_i$ for $1\leq k\leq \delta n$. Probabilities of events involving edge weights of 
$\pi$, unless specfically mentioned, will be assumed to be conditioned on ``$\{A(\pi) \leq 
\lambda\}$'' throughout the remainder of this section. Now divide $\pi$ into edge-disjoint 
subpaths of length $\ell$ (with the last subpath of length possibly less than $\ell$ in the case 
$\ell$ does not divide $\delta n$) and denote the $k$-th subpath by $b_k^{\pi}$ for $1 \leq k 
\leq \delta n / \ell$. Call any such subpath a downcrossing if it is $(\lambda, 1)$-light. Let 
$D_{k, \eta', n}^{\pi}$ be the event that $b_k^\pi$ is a downcrossing. The following 
well-known result about exponential random variables (see, e.g., \cite[Theorem 6.6]{Dasgupta2011}) 
will be very useful.

\begin{lemma}
\label{Beta}
Let $W_1, W_2, \ldots, W_N$ be i.i.d.\ exponential random variables with mean $1/\theta$,  and let 
$S_k = \sum_{i=1}^k W_i$ for $1\leq k\leq N$. Then the random vector $(\frac{W_1}{S_N},\ldots, 
\frac{W_{N-1}}{S_N})$ follows $\mathrm{Dirichlet}(\mathbf{1}_{N})$ distribution, $S_N$ follows 
$\mathrm{Gamma}(N;\theta)$ distribution, and they are independent of each other. Here 
$\mathbf{1}_{N}$ is the $N$-dimensional vector whose all entries are $1$.
\end{lemma}
We will also require the following simple lemma which we prove for sake of completeness. 
\begin{lemma}
\label{Azuma}
Let $Z_1, Z_2, \ldots, Z_N$ be i.i.d.\  exponential random variables with mean 1 and let  $S_N = 
\sum_{i = 1}^N S_N$. Then
\begin{align}
\mathbb{P}(S_N \geq N + \alpha) &\leq \mathrm{e}^{-\alpha^2/4N} \mbox{ for all } 0<\alpha \leq (2 
- \sqrt{2})N\,, \label{Azuma1}\\
\mathbb{P}(S_N \leq N - \alpha) &\leq \mathrm{e}^{-\alpha^2/2N}\,, \mbox{ for all } \alpha>0\,. 
\label{Azuma2}
\end{align}
\end{lemma}
\begin{proof}
By Markov's inequality, we get that for any $\alpha > 0$ and $0 < \theta < 1$,
$$\mathbb{P}(S_N \geq N + \alpha) = \mathbb P(e^{\theta S_N} \geq e^{\theta (N+ \alpha)}) \leq 
\tfrac{\mathrm{e}^{-\theta N - \theta \alpha}}{(1-\theta)^N}.$$
When $ \theta \leq 1 - 1/\sqrt{2}$, the right hand side is bounded above by $\mathrm{e}^{N  
\theta^2 - \alpha\theta}$. So setting $\theta = \alpha / 2N$ yields \eqref{Azuma1} as long as $0 < 
\alpha / 2N \leq 1 - 1/\sqrt{2}$. One can prove \eqref{Azuma2} in the same manner.
\end{proof}
 
As hinted in the introduction, let us begin with the intent to prove that the number of 
downcrossings along the first half of $\pi$ (or any fraction of it) dominates a Binomial random 
variable $\mathrm{Bin}(\delta n / 2\ell, p)$ for some positive, absolute constant $p$. So 
essentially we need to prove that a subpath $b_k^{\pi}$ can be a downcrossing with probability $p$ 
regardless of the first $(k-1)\ell$ edges of $\pi$ that precede it. Now conditional distribution 
of $X_{(k-1)\ell + 1}, X_{(k-1)\ell + 2}, \cdots, X_{\delta n}$ given $X_1, X_2, \cdots, 
X_{(k-1)\ell}$ and $A(\pi) \leq \lambda$ is essentially the the distribution of $X_{(k-1)\ell + 
1}, X_{(k-1)\ell + 2},$ $\cdots, X_{\delta n}$ conditioned on $\sum_{i = (k - 1)\ell + 1}^{\delta 
n}X_i \leq \lambda \delta n - S_{(k - 1)\ell}$. On the other hand we get from Lemma~\ref{Beta} 
that conditional mean and variance of $W(b_k^\pi)$ given $S_{\delta n} - S_{(k-1)\ell} = \mu 
(\delta n - (k-1)\ell)$ are $\mu \ell$ and $\mu^2 \ell(1 + o(1))$ respectively for all $\mu > 0$ 
and $k \leq \delta n / 2$. Hence it is plausible to expect that probability of the event 
$\{W(b_k^{\pi}) \leq \Lambda_k(\ell - C\sqrt{\ell})\}$ conditional on any set of values for $X_1, 
X_2, \cdots, X_{(k - 1)\ell}$ is bounded away from 0 for large $\ell$ and $n$, where $\Lambda_k 
= \Lambda_k^\pi = (S_{\delta n} - S_{(k-1)\ell})/(\delta n - (k-1)\ell)$ and $C$ is some positive 
number. Let us denote the event $\{W(b_k^{\pi}) \leq \Lambda_k(\ell - C\sqrt{\ell})\}$ by $A_{k, 
\eta', \delta, n}^{C, \pi}$. Thus it seems more immediate to prove the stochastic domination for 
number of occurrences of $A_{k, \eta', n}^{C, \pi}$'s which, for the time being, can be treated as 
a ``proxy'' for the number of downcrossings. The formal statement is given in the next lemma where 
we use $6$ as the value of $C$ since this allows us to avoid unnecessary named variables and also 
suits our specific needs for the computations carried out at the end of this section.

\begin{lemma}
\label{drop_prob}
Let $N_{\eta', n}^\pi$ be the number of occurrences of events $A_{k, \eta', n}^{\pi} = 
A_{k, \eta', n}^{6, \pi}$ for $1 \leq k \leq \delta n / 2\ell$. Then for any $0 < \eta' < \eta_0$ 
where $\eta_0$ is a positive, absolute constant and any $0 < \delta_0 < 1$ there exists a positive 
integer $n_d = n_d(\delta_0, \eta')$ and an absolute constant $c > 0$ such that for all $\delta 
\geq \delta_0$ and $n \geq n_d $ the conditional distribution of $N_{\eta',n}^\pi$ given 
$\{A(\pi) \leq \lambda\}$ stochastically dominates the binomial distribution $\mathrm{Bin}(\delta 
n /2\ell, c)$. 
\end{lemma}

\begin{proof}
Notice that it suffices to prove that there exist positive absolute constants $\ell_0, c$ such 
that uniformly for $\mu > 0$, $\ell \geq \ell_0$ and large $L$ (given $\ell$)
$$\mathbb{P}(S_{\ell} \leq \tfrac{S_L}{L}(\ell - 6\sqrt{\ell}) | S_L = \mu L)\geq c\,.$$  To this 
end, we see that for $L > \ell$
\begin{eqnarray}
\mathbb{P}(S_{\ell} \leq \tfrac{S_L}{L}(\ell - 6\sqrt{\ell}) | S_L = \mu L)  = \mathbb{P}
(\tfrac{S_{\ell}}{S_L} \leq (\ell - 6\sqrt{\ell})/L | S_L = \mu L) = \mathbb{P}(\tfrac{S_{\ell}}
{S_L} \leq (\ell - 6\sqrt{\ell})/L) \,,\label{ratio_prob} 
\end{eqnarray}
where the last equality follows from Lemma~\ref{Beta}. Since distribution of $\frac{S_{\ell}}
{S_L}$ does not depend on the mean of the underlying $X_j$'s, we can in fact assume that $X_j$'s 
are i.i.d.\ exponential variables with mean 1 for purpose of computing \eqref{ratio_prob}. By 
\eqref{Azuma2}, we have 
$$\mathbb{P}\big(S_L/L \leq 1 - 1/(2\sqrt{\ell})\big) \leq \mathrm{e}^{-L / 8 \ell}\,.$$
 So for $\ell - 6\sqrt{\ell} > 0$, we get
\begin{equation}
\label{prob_ineq}
\mathbb{P}(S_{\ell} \leq \tfrac{S_L}{L}(\ell - 6\sqrt{\ell})) 
\geq \mathbb{P}(S_{\ell} \leq \ell - 6.5\sqrt{\ell}) - \mathrm{e}^{-L / 8 \ell}.
\end{equation}
By central limit theorem there exist absolute numbers $\ell_0, c'>0$ such that $\mathbb{P}
(S_{\ell} \leq \ell - 6.5\sqrt{\ell}) \geq c'$ for $\ell \geq \ell_0$. Hence from 
\eqref{prob_ineq} it follows that for any $\ell \geq \ell_0$ there exists $L_0 = L_0(\ell)$ such 
that the right hand side of \eqref{ratio_prob} is at least $c = 0.99c'$ for $L \geq L_0$.
\end{proof}
Now what remains to show is that the number of downcrossings $\tilde N_{\eta', n}^\pi$ 
along $\pi$ is bigger than $N_{\eta', n}^\pi$ with high probability. Notice that the 
occurrence of $A_{k, \eta', n}^{\pi} \setminus D_{k, \eta', n}^{\pi}$ implies that 
$\Lambda_k$ must be ``significantly'' above $\lambda$. But that can only be caused by a 
substantial drop in $S_k$ for some $1 \leq k \leq \delta n /2$, an event that occurs with small 
probability.
\begin{lemma}
\label{slope_rise}
Denote by $E_{\eta', n}^\pi$ the event that $\Lambda_k$ is more than $\lambda + \sqrt{\eta'}$ for 
some $1 \leq k \leq \frac{\delta n}{2\ell}$. Then for any $0 < \eta' < 1/4$ and $0 < \delta_0 < 1$ 
there exists a positive integer $n_s = n_s(\delta_0, \eta')$ such that,
\begin{equation}
\label{Ineq2}
\mathbb{P}(E_{\eta', n}^\pi| A(\pi) \leq \lambda) \leq 2n\mathrm{e}^{-\delta n\eta' / 16} \mbox{ 
for all } \delta \geq \delta_0 \mbox{ and }n \geq n_s\,.
\end{equation}
\end{lemma}
\begin{proof}
For  $1 \leq k \leq \delta n/2\ell$, let $\ell_k = (k - 1)\ell$, $n_s = \lceil 2\ell / \delta_0 
\rceil$ and $E_{k, \eta', n}^\pi = \{\Lambda_k\geq \lambda + \sqrt{\eta'}\}$. Assume $n \geq n_s$ 
so that $\delta n / 2\ell \geq 1$. On $E_{k, \eta', 
n}^\pi$, we have
\begin{eqnarray*}
\frac{S_{\ell_k}}{S_{\delta n}} \leq  \frac{\ell_k S_{\delta n} / \delta n - \sqrt{\eta'}(\delta n 
- \ell_k)}{S_{\delta n}} \leq  \frac{\ell_k}{\delta n} - \frac{\sqrt{\eta'}(\delta n - \ell_k)}
{\delta n}\,,
\end{eqnarray*}
where the last inequality holds since we are conditioning on $S_{\delta n} \leq \lambda \delta n$ 
and $\lambda \leq 1$ when $\eta' < 1/4$ (recall that $\eta < \eta'$). Therefore, we get
\begin{equation}\label{slope_change}
\mathbb{P}(E_{k,\eta',n}^\pi | A(\pi) \leq \lambda) 
\leq \mathbb{P}(S_{\ell_k} \leq \tfrac{S_{\delta n}}{\delta n}\big(\ell_k - \sqrt{\eta'}(\delta n 
- \ell_k)))
\end{equation}
Now we evaluate the right hand side of \eqref{slope_change}. Analogous to \eqref{ratio_prob} in 
the proof of Lemma~\ref{drop_prob}, we can assume without loss of generality  that $X_1, 
X_2, \ldots X_L$ are i.i.d.\ exponential variables with  mean 1.
It is routine to check that $$(1+\sqrt{\eta'}/2)\times \big(\ell_k - \sqrt{\eta'}(\delta n - 
\ell_k)\big) \leq \ell_k - \sqrt{\eta'}\delta n/4 \,, \mbox{ for all } 1\leq k \leq \delta 
n/2\ell\,.$$
Thus, for all $1 \leq k \leq \delta n/2\ell$ we get
\begin{eqnarray*}
\mathbb{P}\Big(S_{\ell_k} \leq \frac{S_{\delta n}}{\delta n}\big(\ell_k - \sqrt{\eta'}(\delta n - 
\ell_k)\big)\Big) 
& \leq & \mathbb{P}\Big(S_{\ell_k} \leq \ell_k - \sqrt{\eta'}\delta n/4 \Big) + 
\mathbb{P}\Big(\frac{S_{\delta n}}{\delta n} \geq 1 + \sqrt{\eta'}/2\Big)\\
& \leq & \mathrm{e}^{-\delta n \eta' / 16} + \mathrm{e}^{-\delta n \eta'/ 16},
\end{eqnarray*}
where the second inequality follows from \eqref{Azuma2} and \eqref{Azuma1} respectively. Combined 
with \eqref{slope_change}, it gives that
\begin{equation*}
\mathbb{P}(E_{k,\eta',n}^\pi|A(\pi) \leq \lambda) \leq 2\mathrm{e}^{- \delta n \eta'/16}\,, \mbox{ 
for all }1 \leq k \leq \delta n/2\ell\,.
\end{equation*}
An application of a union bound over $k$ completes the proof of the lemma.
\end{proof}

\begin{proof}[Proof of Theorem~\ref{Prop}: upper bound]
Assume that $\eta' < 1/4 \wedge \eta_0$ where $\eta_0$ is same as given in the statement of 
Lemma~\ref{drop_prob}. Fix a $\delta_0 = \delta_0(\eta')$ in $(0,1)$ and let $n_0 = n_0(\delta_0, 
\eta') = n_d(\delta_0, \eta') \vee n_s(\delta_0, \eta')$, where $n_d, n_s$ are as stated in 
Lemmas \ref{drop_prob} and \ref{slope_rise} respectively. In the remaining part of this section we 
will assume that $n \geq n_0$ and $\delta \geq \delta_0$, so that Lemmas \ref{drop_prob} and 
\ref{slope_rise} become applicable. Now let $\pi$ be a path with length $\delta n$. From 
Lemma~\ref{slope_rise} we get that with large probability $\Lambda_k \leq \lambda + \sqrt{\eta'}$ 
for all $k$ between 1 and $\delta n /2\ell$. But it takes a routine computation to show that 
$A_{n,k,\eta'}^\pi \setminus \{\Lambda_k \leq \lambda + \sqrt{\eta'}\}\subseteq  D_{n,k,
\eta'}^\pi$ when $\eta'$ is small. Thus $\tilde N_{\eta', n}^\pi \geq N_{\eta', n}^\pi$ except on 
$E_{\eta', n}^\pi$. Consequently Lemma~\ref{drop_prob} allows us to use binomial distribution to 
bound quantities like $\mathbb{P}(\tilde N_{\eta', n}^\pi \leq x)$ with a ``small error term'' 
caused by the rare event $E_{\eta', n}^\pi$. Formally, 

\begin{eqnarray*}
\mathbb{P}\Big(\tilde N_{\eta' , n}^\pi \leq 2\mathbb{E}N_{\eta'} | A(\pi) \leq 
\lambda\Big) 
&\leq &  \mathbb{P}\Big(N_{\eta' , n}^\pi \leq 2\mathbb{E}N_{\eta'} | A(\pi) \leq 
\lambda\Big) + \mathbb{P}\Big(E_{\eta', n}| A(\pi) \leq \lambda\Big)\\
& \leq & \mathbb{P}\Big(N_{\eta' , n}^\pi \leq 2\mathbb{E}N_{\eta'}|A(\pi) \leq \lambda\Big) + 
2n\mathrm{e}^{-\delta n \eta'/16}\,,
\end{eqnarray*}
where the last inequality follows from Lemma~\ref{slope_rise}.
Therefore, by Lemma~\ref{drop_prob}, we get that
\begin{equation}
\mathbb{P}\Big(\tilde N_{\eta', n}^\pi \leq 2\mathbb{E}N_{\eta'} | A(\pi) \leq 
\lambda\Big) \leq  \mathbb{P}\Big(\mathrm{Bin}(\delta n / 2\ell, c) \leq 2\mathbb{E}N_{\eta'}\Big) 
+  2n\mathrm{e}^{-\delta n \eta'/16} \,.\label{second_bd}
\end{equation} 
Next let us define a new event as 
$$\Xi_{\eta,\delta_0,n} = \mbox{$\bigcup_{k \geq \delta_0 n}$ $\bigcup_{\pi \in 
\Pi_k}$}\big\{\tilde N_{\eta', n}^\pi \geq 2\mathbb{E}N_{\eta'}, A(\pi) \leq \lambda \big\}.$$
So $\Xi_{\eta,\delta_0,n}$ is the event that there exists a $\lambda$-light path $\pi$ with 
$len(\pi) \geq \delta_0 n$ and which contains at least $2\mathbb{E}N_{\eta'}$ many 
downcrossings. Thus occurrence of $\Xi_{\eta,\delta_0,n}$ implies that $N_{\eta'} \geq 
2\mathbb{E}N_{\eta'}$ which has small probability owing to \eqref{concentration_2}. On the other 
hand if $\Xi_{\eta,\delta_0,n}$ does not occur, $L(n, \lambda) \geq \delta_0 n$ implies the 
existence of a $\lambda$-light path of length at least $\delta_0 n$ that has no more than 
$2\mathbb{E}N_{\eta'}$ many downcrossings. Formally,

\begin{align}
\mathbb{P}&\big(L(\lambda, n) \geq \delta_0 n \big) =  \mathbb{P}\big(\Xi_{\eta,\delta_0,n}\big) + 
\mathbb{P}\big(\{L(\lambda, n) \geq \delta_0 n\} 
\setminus \Xi_{\eta,\delta_0,n}\big)\nonumber \\
& \leq  \mathbb{P}\big(N_{\eta'} \geq 
                                                    2\mathbb{E}N_{\eta'}\big) +  
                                    \mathbb{P}\Big(\mbox{$\bigcup_{k \geq \delta_0 
                                                    n}$ $\bigcup_{\pi \in \Pi_k}$}\big\{\tilde 
                                                    N_{\eta', n}^\pi \leq 2\mathbb{E}N_{\eta'}, 
                                                    A(\pi) \leq \lambda \big\}\Big) \nonumber\\
                                                    & \leq  o(1) +  
                                    \mbox{$\sum_{k \geq \delta_0 n}$ $\sum_{\pi \in 
                                    \Pi_k}$}\mathbb{P}\Big(\tilde N_{\eta', n}^\pi \leq 
                                    2\mathbb{E}N_{\eta'} | A(\pi) \leq
                                    \lambda\Big)\mathbb{P}\big(A(\pi) \leq \lambda\big)\,.
                                    \label{break-up}
\end{align}
Now choose $\delta_0 = \delta_0(\eta')$ such that 
\begin{equation}
\label{Ineq3}
\delta_0 n \eta' c /4 = 2\mathbb{E}N_{\eta'}\,.
\end{equation}
Since $1 / \ell \geq \eta'$, we then get from Binomial concentration that for $\delta \geq 
\delta_0$,
$$\mathbb{P}\Big(\mathrm{Bin}(\delta n / 2\ell, c) \leq 2\mathbb{E}N_{\eta'}\Big) \leq 
\mathrm{e}^{-\delta n\eta'c^2/ 16}\,.$$
Plugging this into \eqref{second_bd} we have
\begin{equation*}
\mathbb{P}\Big(\tilde N_{\eta', n}^\pi \leq 2\mathbb{E}N_{\eta'} | A(\pi) \leq 
\lambda\Big) \leq 2n\mathrm{e}^{- len(\pi)\eta'/16} + \mathrm{e}^{- len(\pi)\eta'c^2/ 16}\,,
\end{equation*}
whenever $len(\pi) \geq \delta_0 n$. A straightforward computation using 
\eqref{gamma_density} yields
$$\mbox{$\sum_{\pi \in \Pi_k}$}\mathbb{P}\big(A(\pi) \leq 
\lambda\big) \leq  \frac{n}{\sqrt{2\pi k}}\mathrm{e}^{\mathrm{e} k \eta}\,.$$

The last two displays and \eqref{break-up} together imply that
\begin{equation}
\label{break_up2}
\mathbb{P}\big(L(\lambda, n) \geq \delta_0 n \big) \leq o(1) + \mbox{$\sum_{k \geq \delta_0 n}$}
(2n\mathrm{e}^{- k\eta'/16} + \mathrm{e}^{- k\eta'c^2/ 16})\mathrm{e}^{\mathrm{e} k 
\eta}\frac{n}{\sqrt{2\pi k}}\,.
\end{equation}
Setting $\eta' = 32\mathrm{e}\eta / c^2$ we get from \eqref{break_up2},
$$\mathbb{P}\big(L(n, \lambda)\geq \delta_0  n\big) = o(1)\,.$$
It remains to be checked whether $\delta_0$ obtained from \eqref{Ineq3} has the correct functional 
form as in \eqref{main_prop}. To this end recall from \eqref{first_bnd} that
$$2\mathbb{E}N_{\eta'} \leq 3\alpha \mathrm{e}^{\mathrm{e}\eta/\eta'} \sqrt{\eta'} 
\mathrm{e}^{-\mathrm{e}/\sqrt{\eta'}}n\,,$$
where $\eta$ is small enough so that $C_0(\eta)$ in \eqref{first_bnd} is less than $3/2$. Hence 
$\delta_0 \leq \mathrm{e}^{-C_2 / \sqrt{\eta}}$ for some absolute constant $C_2$ when $\eta$ is 
small.

\end{proof}

\section{Proof of the lower bound}
\label{sec-lower}

\subsection{Existence of a large number of vertex-disjoint light paths}
As we mentioned in the introduction, the proof of lower bound is divided into two steps. In the 
first step we split the vertices into two parts and show that there exist a large number of 
short (i.e. of $O(1)$ length) vertex-disjoint $\lambda$-light paths containing vertices from only 
one part. In the second step we use vertices in the other part as ``links'' to connect a 
subcollection of the short paths obtained from step 1 into a long (i.e. of $\Theta(n)$ length) and 
light path. The current and next subsections are devoted to these two steps in respective order.
\par

In light of the preceding discussion, let us first select a complete subgraph $\mathcal W^*_n$ of 
$\mathcal W_n$ containing $n_* = n_{*; \eta, \zeta_1} = (1 - \zeta_1 \eta)n$ vertices where $\eta, 
\zeta_1 \in (0, 1)$. To be specific we can order the vertices of $\mathcal W_n$ in some arbitrary 
way and define $\mathcal W^*_n$ as the subgraph induced by ``first'' $n_*$ vertices. It will be 
shown that there are substantially many short and light paths that can be formed with the vertices 
in $V(\mathcal W_n^*)$. We will in fact require slightly more from a path than just being 
$\lambda$-light. For $\pi \in \Pi_\ell$ and some $\zeta_2 > 0$, define

\begin{equation}
G_{\pi} = G_{\pi; \eta, \zeta_2} = \Big \{\lambda \ell - 1 \leq W(\pi) \leq \lambda \ell, 
M(\pi) \leq (\zeta_2/\sqrt{\eta}).(W(\pi)/\lambda \ell)\Big \}\,,
\label{good_event}
\end{equation}

where $M(\pi)$ is the maximum deviation of $\pi$ away from the linear interpolation between the 
starting and ending edges, formally given by
\begin{equation}\label{eq-def-M}
M(\pi) = \sup_{1\leq k \leq \ell}\mid\mbox{$\sum_{i = 1}^{k}$}W_{e_i} - \tfrac{k}
{\ell}W(\pi)\mid\,.
\end{equation}
A similar class of events were considered in \cite{Ding13, MW13} in order for second moment 
computation. As the authors mentioned in these papers, the factor $W(\pi)/\lambda \ell$ provides 
some technical ease in view of the following property which is a consequence of Lemma~\ref{Beta}:
\begin{equation}
\mathbb{P}(M(\pi) \leq (\zeta_2/\sqrt{\eta}).(W(\pi)/\lambda \ell)\mbox{ }|\mbox{ }W(\pi) = w) 
\equiv \mbox{constant for all }w > 0.\label{cond_Property}
\end{equation}
Call a path $\pi\in \Pi_\ell$ \emph{good} if $G_{\pi}$ occurs. Since we are only interested in 
good paths whose vertices come from $V(\mathcal W^*_n)$, we need some related notations. For $\ell 
\in \mathbb N$, denote by $\Pi_\ell^*  = \Pi_{\ell; \eta, \zeta_1}^*$ the set of all paths of 
length $\ell$ in $\mathcal{W}_n^*$ and by $N^*_\ell = N^*_{\ell; \eta, \zeta_1, \zeta_2}$ the 
total number of good paths in $\Pi_\ell^*$, i.e., $N^*_\ell = \sum_{\pi \in 
\Pi_\ell^*}\mathbf{1}_{G_{\pi}}$. In order to carry out second 
moment analysis of $N^*_\ell$ we need to control the correlation between $\mathbf{1}_{G_{\pi}}$ 
and $\mathbf{1}_{G_{\pi'}}$ where $\pi$, $\pi' \in \Pi_\ell^*$. It is plausible that such 
correlation depends on the number of common edges between $\pi$ and $\pi'$ and in fact bounding 
the correlation in terms of the number of common edges was sufficient for proving 
\eqref{concentration_1} in Section \ref{sec-upper}. But in this case we need an additional 
measurement instead of just $|E(\pi) \cap E(\pi')|$. This is discussed in detail in \cite{Ding13, 
MW13} and some of their results will be used. Let $\pi$ be a path in $\Pi_\ell ^*$ and $S 
\subseteq E(\pi)$. A segment of $\pi$ is called an $S$-component or a component of $S$ if it is a 
maximal segment of $\pi$ whose all edges belong to $S$. Notice that $S$-components can be defined 
solely in terms of $S$. For two paths $\pi$ and $\pi'$, define a functional $\theta(\pi, \pi')$ to 
be the number of $S$-components where $S = E(\pi) \cap E(\pi')$. As $\pi$ and $\pi'$ are self-avoiding, $\theta(\pi, \pi')$ is basically the number of maximal segments shared between $\pi$ and 
$\pi'$. We refer the readers to Figure \ref{fig:fig0} for an illustration.

\begin{figure}[!ht]
 \centering\includegraphics[width=0.8\textwidth, keepaspectratio]{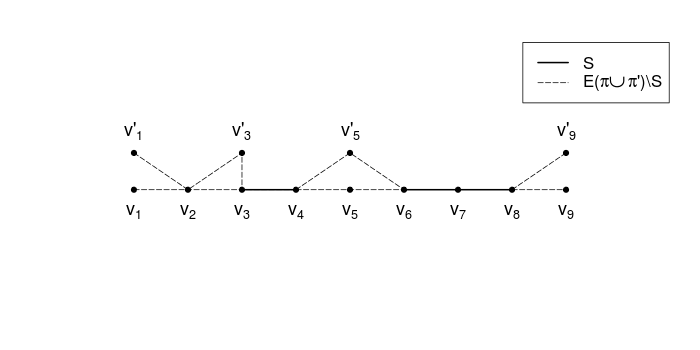}
 \caption{{\bf Components of the set of edges common to two paths.} In this figure the sequences  
of vertices $v_1, v_2, v_3, v_4, v_5, v_6, v_7, v_8, v_9$ and $v_1', v_2, v_3', v_3, v_4, v_5', 
v_6,  v_7, v_8, v_9'$ define the paths $\pi$ and $\pi'$ respectively. The dark edges belong to $S 
= E(\pi) \cap E(\pi')$. Here $\theta(\pi, \pi') = 2$ with the segments $(v_3, v_4)$ and 
$(v_6, v_7, v_8)$ being the two $S$-components.}
\label{fig:fig0}
\end{figure}

The following result (\cite[Lemma 2.9]{Ding13}) relates cardinality of $V(S)$, the union of all 
endpoints of edges in $S = E(\pi) \cap E(\pi')$, to $\theta(\pi, \pi')$ and $|S|$.

$$ |V(S)| = |S| + \theta(\pi, \pi')\,.$$

The pair $\big(\theta(\pi, \pi'), |E(\pi) \cap E(\pi')|\big)$ turns out to be sufficient for 
bounding the correlation between $\mathbf{1}_{G_{\pi}}$ and $\mathbf{1}_{G_{\pi'}}$ from above. 
Consequently it makes sense to partition $\Pi_\ell^*$ based on the value of this pair. More 
formally for $\pi \in \Pi_\ell^*$ and integers $i \leq j$, define the set $A_{i,j}$ as
\begin{equation} \label{eq-def-A-i-j}
A_{i,j} \equiv A_{i,j}(\pi) = \{\pi' \in \Pi_\ell^*: \theta(\pi,\pi') = i, |E(\pi)\cap E(\pi')| = 
j\}.
\end{equation}
We need a number of lemmas from \cite{Ding13}. 
\begin{lemma}
\label{up_bnd_lemma}(\cite[Lemma 2.10]{Ding13})
For any $1 \leq \ell \leq n_*$ and any $\pi \in \Pi_\ell^*$, we have that 
for any positive integers $i \leq j$
\begin{eqnarray*}
\label{up_bnd}|A_{i,j}(\pi)| \leq  \tbinom{\ell + 1}{2i} \tbinom{n_* - i - j}{\ell + 1 - i - 
j}2^i(\ell + 1 -j)!\leq  \ell^{3i}n_*^{\ell + 1 - i - j}.
\end{eqnarray*}
\end{lemma}

\begin{lemma}( \cite[Lemma 2.3]{Ding13})
\label{cond_dev1}
Let $Z_i$ be i.i.d.\ exponential variables with mean $\theta > 0$ for $1 \leq i \leq \ell$. For 
$1/4 \leq \rho \leq 4$, consider the variable
\begin{equation}
\label{max_dev_new}
M_\ell = \sup_{1 \leq k \leq \ell}\mid\mbox{$\sum_{i=1}^k$} Z_i - \rho k \mid.
\end{equation}
Then there exist absolute constants $c^*$, $C^* > 0$ such that for all $r \geq 1$ and $\ell \geq 
r^2$,
\begin{equation*}
\mathrm{e}^{-C^*\ell/r^2} \leq \mathbb{P}(M_\ell \leq r \mid \mbox{$\sum_{i=1}^\ell$} Z_i = \rho 
\ell ) \leq \mathrm{e}^{-c^*\ell/r^2}\,.
\end{equation*}
\end{lemma}

\begin{lemma}\cite[Lemma 3.2]{Ding13}
\label{cond_dev}
Let $Z_i$ be i.i.d.\ exponential variables with mean $\theta > 0$ for $i \in \mathbb{N}$. Consider 
$1 \leq r \leq \sqrt{\ell}$ and the integer intervals $[a_1, b_1], [a_2, b_2], \cdots, [a_m, b_m]$ 
such that $1 \leq a_1 \leq b_1 \leq a_2 \leq \cdots \leq a_m \leq b_m \leq \ell$ and $q = 
\sum_{i=1}^m(b_i - a_i + 1) \leq \ell - 1$. Let $1/4 \leq \rho \leq 1$ and $M_\ell$ be defined as 
in the previous lemma. Also write $A = \cup_{i = 1}^m [a_i,b_i] \cap \mathbb{N}$ and $p_\ell 
= \mathbb{P}(M_\ell \leq r | \sum_{i=1}^\ell Z_i = \rho \ell)$. Then for all $z_j$ such that
$$\mbox{$\sum_{j = a_i}^{b_i}$} z_j - \rho(b_i - a_i + 1) \leq 2r\,,$$
we have 
\begin{equation}
\label{cond_dev_ineq}
\mathbb{P}(M_\ell \leq r | \mbox{$\sum_{i = 1}^{\ell}$}Z_i = \rho \ell, Z_j = z_j \mbox{ for all 
}j \in A) \leq C_3r\sqrt{q \wedge (\ell - q)} p_\ell 10^{100mr}\mathrm{e}^{C^*q/r^2}\,,
\end{equation}
where $C^*$ is the constant from Lemma~\ref{cond_dev1} and $C_3 > 0$ is an absolute constant.
\end{lemma}
\begin{remark}
(1) Notice that the bounds in Lemma~\ref{cond_dev1} and \ref{cond_dev} do not depend on the 
particular mean of $Z_i$'s due to Lemma~\ref{Beta}. (2) Although the bounds on $p_\ell$ in 
Lemma~\ref{cond_dev1} do not contain any $\rho$ (as it was restricted to a bounded interval), 
$p_\ell$ actually depends on $r$ only through the ratio $r/\rho$. This follows from an application 
of Lemma~\ref{Beta} with little manipulation. (3) Lemma~\ref{cond_dev} is same as Lemma 3.2. in 
\cite{Ding13} except that in the latter $q$ is restricted to be at most $\ell - 10r$. But we can 
easily extend this to all $q \leq \ell - 1$. To see this assume $\ell - 1 \geq q \geq \ell - 10r$. 
Then the right hand side in \eqref{cond_dev_ineq} becomes at least $C_3 p_\ell 
\mathrm{e}^{C^*\ell/r^2}\mathrm{e}^{-10C^*/r}$. Now from Lemma~\ref{cond_dev1} we get $p_\ell 
\mathrm{e}^{C^*\ell/r^2} \geq 1$. So the right hand side in \eqref{cond_dev_ineq} is bigger than 
$C_3\mathrm{e}^{-10C^*}$ whenever $\ell - 1 \geq q \geq \ell - 10r$. Increasing $C_3$ if necessary 
we can make this number bigger than 1 and thus Lemma~\ref{cond_dev} follows.
\end{remark}
By second moment computation, we can hope to show that $N^*_\ell \sim \E N^*_\ell$ with high 
probability. Then the main challenge is to prove that a large fraction of the good paths are 
mutually vertex-disjoint with high probability. To this end, we consider a graph $\mathcal{G}_n$ 
where each vertex corresponds to a good path in $\mathcal W_n^*$ and an edge is present whenever 
the corresponding paths intersect at one vertex at least. Thus the presence of a large number 
of vertex disjoint good paths in $\mathcal W^*_n$ is equivalent to the existence of a large 
independent subset (i.e., a subset that has no edge among them) in the graph $\mathcal G_n$. The 
following simple lemma is sometimes referred to as Tur\'{a}n's theorem, and can be proved simply 
by employing a greedy algorithm (see, e.g., \cite{Erdos70}).
\begin{lemma}
\label{relation}
Let $G = (V,E)$ be a finite, simple graph with $V \neq \emptyset$. Then $G$ contains an 
independent subset of size at least $|V|^2 / (2|E| + |V|)$. Notice $2|E|$ is the total degree of 
vertices in $G$.
\end{lemma}
In light of Lemma~\ref{relation}, we wish to show that with high probability the total degree of 
vertices in $\mathcal{G}_n$ is not big relative to $|V(\mathcal{G}_n)|$.
For this purpose, it is desirable to show that the typical number of good paths that intersect 
with a fixed good path $\pi \in \Pi_\ell^*$ is not big. Thus, we need to estimate $\sum_{\pi' \in 
\Pi_{\ell,\pi}^*}\mathbb{P}(G_{\pi'}|G_{\pi})$ where $\Pi_{\ell,\pi}^*$ is the collection of all 
paths $\pi'$ in $\mathcal W_n^*$ sharing at least one vertex with $\pi$. Drawing upon the 
discussions preceding \eqref{eq-def-A-i-j}, we will first estimate $\mathbb{P}(G_{\pi'}|G_{\pi})$ 
for a specific value of the pair $(\theta(\pi, \pi'), |E(\pi)\cap E(\pi')|)$. Our next lemma is 
very similar to Lemma~3.3 in \cite{Ding13}.
\begin{lemma}
\label{conditional_prob}
Let $\pi \in \Pi_\ell^*$ and $\pi' \in A_{i,j}$ with $1 \leq i \leq j \leq \ell$. Then there exist 
absolute constants $\eta_1, C_4>0$ such that for $0 < \eta < \eta_1$, $\zeta_2 > 1 \vee 
\sqrt{2C^*/\mathrm{e}}$ and $\ell \geq \zeta_2^2/\eta$ we have
\begin{equation}
\label{cond_est}
\mathbb{P}(G_{\pi'}|G_{\pi}) \leq C_4(1 + o(1))\mathbb{P}(G_{\pi})n^j\sqrt{\ell/\eta} 
\mathrm{e}^{-j\eta}\mathrm{e}^{1000\zeta_2 i/\sqrt{\eta}}\,.
\end{equation}
\end{lemma}
\begin{proof}
Denote by $S$ and $S'$ the sets $E(\pi)\cap E(\pi')$ and $E(\pi')\setminus E(\pi)$ respectively. 
By standard calculus, there exists $0<\eta_1 \leq 1$ such that $1 + \mathrm{e}\eta \geq 
\mathrm{e}^{(1 + \mathrm{e}/2)\eta}$ for all $0 < \eta < \eta_1$. Note that $\mathbb{P}(G_{\pi'} 
\mid G_{\pi}) = \mathfrak p_1 \cdot \mathfrak p_2$, where

\begin{eqnarray*}
\mathfrak p_1 &=& \mathbb{P}(\lambda \ell - 1 \leq W(\pi') \leq \lambda \ell\mbox{ }|\mbox{ 
}G_{\pi})\,, \\
                  \mathfrak p_2           &=& \mathbb{P}\big(M(\pi') \leq (\zeta_2 / \sqrt{\eta}).
                  (W(\pi')/\lambda \ell)\mbox{ }|\mbox{ }G_{\pi}, \lambda \ell - 1 \leq W(\pi') 
                  \leq \lambda \ell\big)\,.
\label{decomposition}
\end{eqnarray*}

Since the maximum deviation of a good path from its linear interpolation between starting and 
ending edges is at most $\zeta_2/\sqrt{\eta}$, the weight of an $S$-component, say $s$, is at 
least $W(\pi)|s| / \ell - 2\zeta_2/\sqrt{\eta}$ when $\pi$ is good. Here $|s|$ denotes the number 
of edges in $s$. Adding over all the $\theta(\pi, \pi')$ components of $S$ we get that
$\mbox{$\sum_{e\in S}$}W_e \geq W(\pi)|S| / \ell - 2\theta(\pi, \pi')\zeta_2/\sqrt{\eta} \mbox{ 
on }G_\pi$. As $\pi' \in A_{i, j}$ and weight of a good path is at least $\lambda \ell - 1$, the 
previous inequality implies that on $G_\pi$,
$$\mbox{$\sum_{e\in S}$}W_e \geq \lambda j - 1 - 2i\zeta_2/\sqrt{\eta}\,.$$
Consequently when $j \leq \ell - 1$,
\begin{eqnarray}
\mathfrak p_1&\leq & \mathbb{P}(\mbox{$\sum_{e\in S'}$}W_e \leq \lambda|S'| + 1 + 2i 
\zeta_2/\sqrt{\eta} \mbox{ }|\mbox{ }G_{\pi})\\
&=& \mathbb{P}\big(\mathrm{Gamma}(\ell-j, 1/n) \leq \lambda(\ell - j) + 1 + 2i 
\zeta_2/\sqrt{\eta}\big) \nonumber\\
&\leq & C_4' n^{-(\ell-j)}(\ell-j)^{-1/2}(1 + \mathrm{e}\eta)^{\ell-j} \mathrm{e}^{2i 
\mathrm{e}\zeta_2 / \sqrt{\eta}(1 + \mathrm{e}\eta)}\,, \label{eq-p-1}
\end{eqnarray}
where $C_4'>0$ is an absolute constant and the last inequality used \eqref{gamma_density}. For the 
second term in the right hand side of \eqref{decomposition}, we can apply \eqref{cond_Property} 
and Lemma~\ref{cond_dev} to obtain
\begin{eqnarray}
\mathfrak p_2 \leq  C_3\mathbb{P}\big(M(\pi) \leq \zeta_2/\sqrt{\eta}\mbox{ }|\mbox{ 
}W(\pi)=\lambda l\big)\sqrt{j\wedge (\ell-j)/\eta} 10^{100i\zeta_2 / 
\sqrt{\eta}}\mathrm{e}^{C^*j\eta/\zeta_2^2} \,, \label{cond_prob1}
\end{eqnarray}
when $j \leq \ell - 1$ and $\ell \geq \zeta_2^2 / \eta$ (see the conditions in 
Lemma~\ref{cond_dev}). Using \eqref{cond_Property} again, we get that
\begin{eqnarray*}
\mathbb{P}\big(M(\pi) \leq \zeta_2 / \sqrt{\eta}\mbox{ }|\mbox{ }W(\pi)=\lambda \ell\big) 
& = & \mathbb{P}\big(M(\pi) \leq (\zeta_2 / \sqrt{\eta}).(W(\pi)/\lambda \ell)\mbox{ }|\lambda 
\ell - 1 \leq W(\pi) \leq \lambda \ell\big) \\ 
& = & \mathbb{P}(G_{\pi})/\mathbb{P}(\lambda \ell - 1 \leq W(\pi) \leq \lambda \ell)\\
& = & \mathbb{P}(G_{\pi}) / \mathbb{P}\big(\lambda \ell - 1 \leq \mathrm{Gamma}(\ell, 1/n)\leq 
\lambda \ell\big)\\
&\leq & C_4''(1 + o(1))\mathbb{P}(G_{\pi})\ell!\big(n / \lambda \ell\big)^\ell \,,
\end{eqnarray*}
where $C_4'' > 0$ is an absolute constant and the last inequality follows from 
\eqref{gamma_density}.
Plugging the preceding inequality into \eqref{cond_prob1} and using the fact $\ell! \leq 
\mathrm{e}\sqrt{\ell}(\ell/\mathrm{e})^{\ell}$ (Stirling's approximation)
$$\mathfrak p_2 \leq \mathrm{e}C_3C_4''(1 + o(1))\mathbb{P}(G_\pi)n^\ell \sqrt{\ell(\ell - 
j)/\eta}(1 + \mathrm{e}\eta)^{-\ell}10^{100i\zeta_2 / 
\sqrt{\eta}}\mathrm{e}^{C^*j\eta/\zeta_2^2}\,.$$
Combined with \eqref{eq-p-1}, it yields that
\begin{eqnarray*}
\mathbb{P}(G_{\pi'}|G_{\pi}) &\leq & \mathrm{e}C_3C_4' C_4''\zeta_2(1 + o(1))\mathbb{P}
(G_{\pi})\sqrt{\ell / \eta} n^j (1 + e \eta)^{-j} 10^{100i\zeta_2 / 
\sqrt{\eta}}\mathrm{e}^{C^*j\eta/\zeta_2^2}\mathrm{e}^{2ie\zeta_2 / \sqrt{\eta}(1 + e\eta)} \,.
\end{eqnarray*}
Since $\zeta_2 \geq \sqrt{2C^*/e}$ and $\eta < \eta_1$ we have
\begin{eqnarray*}
\mathbb{P}(G_{\pi'}|G_{\pi}) &\leq & \mathrm{e}C_3C_4' C_4''(1 + o(1))\mathbb{P}
(G_{\pi})n^j\sqrt{\ell/\eta} \mathrm{e}^{-j\eta}\mathrm{e}^{1000\zeta_2 i/\sqrt{\eta}}
\end{eqnarray*}
provided $j \leq \ell - 1$. The case $j = \ell$ can also be easily accommodated. To this end let 
us first compute $\mathbb{P}(G_{\pi})$. It follows from \eqref{gamma_density} and 
Lemma~\ref{cond_dev1} that
\begin{equation*}
\mathbb{P}(G_{\pi}) \geq (1 + o(1))(1 - \mathrm{e}^{-1/\lambda})(\lambda \ell / n)^{\ell} (1 / 
\ell!) \mathrm{e}^{-C^* \ell \eta / \zeta_2^2}\,.
\end{equation*}
 Applying Stirling's formula again, we get that for $\zeta_2 \geq \sqrt{2C^*/e}$ and $\eta < 
 \eta_1$, \begin{equation*}
\mathbb{P}(G_{\pi}) \geq C_4'''(1 + o(1))n^{-\ell}\ell^{-1/2}\mathrm{e}^{\ell\eta}\,,
\end{equation*}
for an absolute constant $C_4'''>0$. Hence, with the choice of $C_4 = 1/C_4''' \vee 
\mathrm{e}C_3C_4' C_4''$ the right hand side of \eqref{cond_est} is at least 1, and thus 
\eqref{cond_est} holds in this case.
\end{proof}
Armed with Lemma~\ref{conditional_prob}, we can now obtain an upper bound on $\sum_{\pi' \in 
\Pi^*_{\ell,\pi}}\mathbb{P}(G_{\pi'}|G_{\pi})$. Similarly we can bound $\sum_{\pi' \in 
\Pi^*_\ell}\mathbb{P}(G_{\pi'}|G_{\pi})$ which is useful for the computation of $\E((N_\ell^*)^2)$ 
in view of the following simple observation:

\begin{equation}
\label{sec_mom_cond}
\E ((N_\ell^*)^2) = \mbox{$\sum_{\pi \in \Pi_\ell^*}$}\mathbb{P}(G_\pi)\mbox{$\sum_{\pi' \in 
\Pi^*_{\ell}}$}\mathbb{P}(G_{\pi'}|G_{\pi}) = \E(N_\ell^*)\mbox{$\sum_{\pi' \in 
\Pi^*_{\ell}}$}\mathbb{P}(G_{\pi'}|G_{\pi})\,,
\end{equation}
where the last equality follows from the fact that $\sum_{\pi' \in 
\Pi^*_{\ell}}\mathbb{P}(G_{\pi'}|G_{\pi})$ is independent of $\pi$.

\begin{lemma}
\label{lemma_crucial1}
Let $0 < \zeta_1 < 1/4$ and let $\zeta_2, \ell, \eta$ satisfy the same conditions as stated in 
Lemma~\ref{conditional_prob}. Then there exists an absolute constant $C_5 > 0$ such that,
\begin{align}
&\mbox{$\sum_{\pi' \in \Pi^*_{\ell,\pi}}$}\mathbb{P}(G_{\pi'}|G_{\pi}) \leq C_5(1 + 
o(1))\mathrm{e}^{1000\zeta_2/ \sqrt{\eta}}\sqrt{\ell^7/\eta}\tfrac{\mathbb{E} N^*_\ell}{n}\,, 
\label{sum1}\\
&\mbox{$\sum_{\pi' \in \Pi^*_{\ell}}$}\mathbb{P}(G_{\pi'}|G_{\pi}) \leq (1 + o(1)) 
\mathbb{E}N^*_\ell \,. \label{sum2}
\end{align}
\end{lemma}
\begin{proof}

By Lemmas \ref{conditional_prob} and \ref{up_bnd_lemma}, we get that for $1 \leq i \leq j \leq 
\ell$,
\begin{eqnarray}
\mbox{$\sum_{\pi' \in A_{i,j}} $}\mathbb{P}(G_{\pi'}|G_{\pi}) &\leq& (1 + o(1))n_*^{\ell + 
1}\mathbb{P}(G_{\pi})\tfrac{\xi(\eta,\ell,i,j, \zeta_1)}{n^i} \leq 
                                                          (1 + 
                                                          o(1))\mathbb{E}N^*_\ell\tfrac{\xi(\eta,
                                                          \ell,i,j, \zeta_1)}
                                                          {n^i}\label{secstage1}\,,
\end{eqnarray}
where $\xi(\eta,\ell,i,j, \zeta_1)$ is a number depending only on $(\eta, 
\ell, i, j, \zeta_1)$ (so in particular, $\xi(\eta,\ell,i,j, \zeta_1)$ does not depend 
on $n$). It is also clear that 
$$\mbox{$\sum_{\pi' \in A_{0, 0}} $}\mathbb{P}(G_{\pi'}|G_{\pi}) \leq \mbox{$\sum_{\pi' \in A_{0, 
0}} $}\mathbb{P}(G_{\pi'}) \leq \E N^*_\ell\,.$$
Combined with \eqref{secstage1}, it yields \eqref{sum2}. It remains to prove \eqref{sum1}. To this 
end, we note that the major contribution to the term $\sum_{\pi' \in \Pi^*_{\ell,\pi}} \mathbb{P}
(G_{\pi'}|G_{\pi})$ comes from those paths $\pi'$ with $\theta(\pi,\pi') = 1$ or $|V(\pi') \cap 
V(\pi)| = 1$. Thus, we revisit \eqref{secstage1} for the case of $i = 1$. By Lemmas 
\ref{conditional_prob} and  \ref{up_bnd_lemma} again, we get that
\begin{eqnarray}
\mbox{$\sum_{1 \leq j \leq \ell} \sum_{A_{1,j}}$}\mathbb{P}(G_{\pi'} | G_{\pi}) 
&\leq & 2C_4 (1 + o(1))\mathrm{e}^{1000\zeta_2 / 
\sqrt{\eta}}\sqrt{\ell^{7}/\eta}n_*^{\ell+1}\mathbb{P}(G_{\pi}) n^{-1} \mbox{$\sum_{1\leq j \leq 
\ell}$} \mathrm{e}^{-j\eta}(1 - \zeta_1 \eta)^{-j} \nonumber\\
& \leq & 2C_4 (1 + o(1))\mathrm{e}^{1000\zeta_2 / \sqrt{\eta}}\sqrt{\ell^{7}/\eta} 
\mathbb{E}N^*_\ell (n(1-\mathrm{e}^{-\frac{\eta}{2}}))^{-1} \nonumber \\
&\leq &  8C_4 (1 + o(1))\mathrm{e}^{1000\zeta_2 / 
\sqrt{\eta}}\sqrt{\ell^{7}/\eta^3}\mathbb{E}N^*_\ell n^{-1}\,, \label{eq-A-1-j}
\end{eqnarray}
where the last two inequalities follow from the facts that $\zeta_1 < 1/4$ and $\mathrm{e}^{-\eta 
/ 2} \leq 1 - \eta/4$ whenever $0 < \eta < 1$. 
We still need to consider paths that share vertices with $\pi$ but no edges. For $1\leq i\leq 
\ell$, define $B_i$ to be the collection of paths which shares $i$ vertices with $\pi$ but no 
edges, i.e., 
\begin{equation*}
B_i = \{\pi' \in \Pi_\ell^* : |V(\pi') \cap V(\pi)| = i, E(\pi') \cap E(\pi) = \emptyset\}\,.
\end{equation*}
We need an upper bound on the size of $B_i$. To this end notice that there are $\tbinom{\ell + 1}
{i}$ many choices for $V(\pi') \cap V(\pi)$ as cardinality of the latter is $i$ and these vertices 
can be placed along $\pi'$ in at most $\tbinom{\ell + 1}{i}i!$ many different ways. Also the 
number of ways we can choose the remaining $\ell + 1 - i$ vertices is at most $n_*^{\ell + 1 - 
i}$. Multiplying these numbers we get
$$|B_i| \leq \tbinom{\ell + 1}{i}^2i!n_*^{\ell + 1 - i}\,.$$
Since the edge sets are disjoint, $\mathbb{P}(G_{\pi'}|G_{\pi}) = \mathbb{P}(G_{\pi})$ for all 
$\pi' \in B_i$ and  $1\leq i \leq \ell$. So we have
\begin{equation}
\mbox{$\sum_{\pi' \in B_{i}} $}\mathbb{P}(G_{\pi'}|G_{\pi}) \leq (1+o(1)) \tbinom{\ell + 1}{i}^2i!
(1 - \zeta_1\eta)^{-i}\tfrac{\mathbb{E}N^*_\ell}{n^i} \leq (8+o(1)) 
\ell^2\tfrac{\mathbb{E}N^*_\ell}{n} \,.\label{secstage2}
\end{equation}
Combined with \eqref{eq-A-1-j}, it completes the proof of \eqref{sum1}.
\end{proof}
We will now proceed with our plan of finding a large independent subset of $\mathcal{G}_n$. For 
any two paths $\pi$ and $\pi'$ in $\Pi^*_{\ell}$, define an event 
$$H_{\pi,\pi'} = H_{\pi,\pi'; \eta, \zeta_2} = \begin{cases}
G_{\pi} \cap G_{\pi'} &\mbox{ if } V(\pi) \cap V(\pi') \neq \emptyset\,,\\
\emptyset, &\mbox{ otherwise}\,.
\end{cases}$$
Writing $N_{\ell}'= N_{\ell; \eta, \zeta_1, \zeta_2}' = \sum_{\pi, \pi' \in 
\Pi_{\ell}^*}\mathbf{1}_{H_{\pi,\pi'}}$, we see that $N'_\ell = 2 |E(\mathcal{G}_n)| + |
V(\mathcal{G}_n)|$. Also notice that $N^*_\ell = |V(\mathcal{G}_n)|$. As an immediate consequence 
of Lemma~\ref{lemma_crucial1}, we can compute an upper bound of $\mathbb{E}N_{\ell}'$ as follows:
\begin{eqnarray}
\mathbb{E}N_{\ell}'  =  \mbox{$\sum_{\pi \in \Pi^*_\ell} $} \mathbb{P}(G_\pi) \mbox{$\sum_{\pi' 
\in \Pi^*_{\ell,\pi}}$}\mathbb{P}(G_{\pi'}|G_{\pi}) \leq C_5(1 + 
o(1))\mathrm{e}^{1000\zeta_2/\sqrt{\eta}}\sqrt{\ell^{7}/\eta^3}\tfrac{(\mathbb{E}N^*_{\ell})^2} 
{n}.\label{expected_edge}
\end{eqnarray}
If $N_\ell^*$ and $N_\ell'$ are concentrated around their respective means in the sense that 
$N_\ell^* = \E N_\ell^*(1 + o(1))$ and $N_\ell' = \E N_\ell'(1 + o(1))$ with high probability, 
then we can use Lemma~\ref{relation} and \eqref{expected_edge} to derive a lower bound on the size 
of a maximum independent subset of $\mathcal G_n$. For this purpose, it suffices to show that 
$\mathbb{E}((N_\ell^*)^2) = (\mathbb{E}N_\ell^*)^2(1 + o(1))$ and $\mathbb{E}((N_\ell')^2) = 
(\mathbb{E}N_\ell')^2(1 + o(1))$. The former has already been addressed by \eqref{sum2} (see 
\eqref{sec_mom_cond}). For the latter we need to estimate contributions from terms like 
$\mathbb{P}(H_{\pi_1,\pi_2} \cap H_{\pi_3,\pi_4})$ in the second moment calculation for $N_\ell'$. 
Our next lemma will be useful for this purpose.
\begin{lemma}
\label{count_vertices}
Let $\pi_1, \pi_2, \pi_3, \pi_4$ be paths in $\Pi_\ell^*$ such that  $|E(\pi_3 \cup \pi_4)| = 
2\ell - j$ and $|E(\pi_1 \cup \pi_2) \cap E(\pi_3 \cup \pi_4)| = j'$ where $0 \leq j \leq \ell$ 
and $1 \leq j' \leq 2\ell - j$. Also assume that $V(\pi_3 \cap \pi_4) \neq \emptyset$. Then,
\begin{equation}
\label{vertex_count}
|V(\pi_3) \cap V(\pi_4)| + |V(\pi_3 \cup \pi_4) \cap V(\pi_1 \cup \pi_2)| \geq j + j' +2\,.
\end{equation}
\end{lemma}
\begin{figure}[!ht]
 \centering\includegraphics[width=0.8\textwidth]{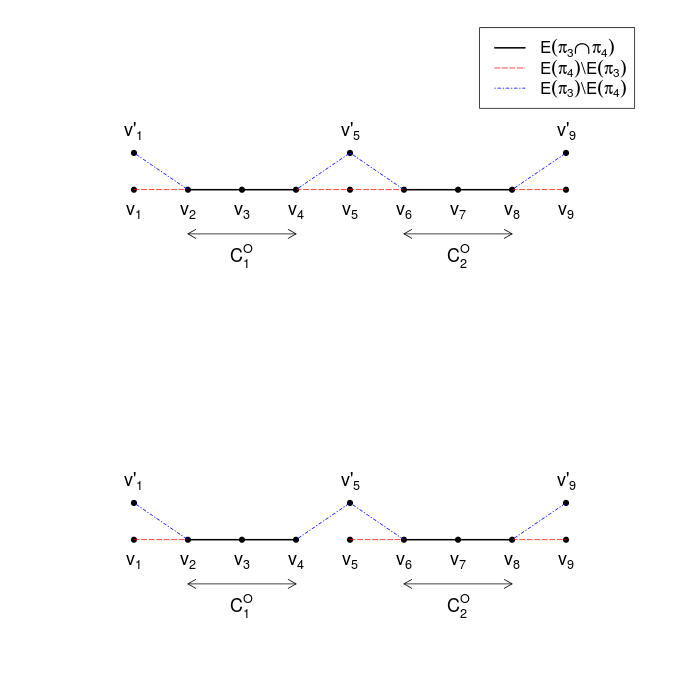}
 \caption{{\bf Removing edges from union of two paths.} In these figures the sequences of vertices 
 $v_1, v_2, v_3, v_4, v_5, v_6, v_7, v_8, v_9$ and $v_1', v_2, v_3, v_4, v_5', v_6, 
 v_7, v_8, v_9'$ define the paths $\pi_4$ and $\pi_3$ respectively. $C^O_1$ and $C^O_2$ are the 
 two connected components of $\pi_3 \cap \pi_4$. In the figure at the top, the vertices $v_4, 
 v_5, v_6, v_5'$ define a cycle. After removing the edge $(v_4, v_5)$ from the only segment in 
 $E(\pi_4) \setminus E(\pi_3)$ between $C^O_1$ and $C^O_2$, we get an acyclic graph displayed at 
 the bottom.}
 \label{fig:fig2}
\end{figure}
\begin{proof}
Suppose the graph $\pi_3 \cap \pi_4$ has exactly $k + 1$ (connected) components namely $C^O_1, 
\cdots, C^O_{k+1}$. Notice that $k$ is nonnegative as $\pi_3 \cap \pi_4 \neq \emptyset$. Since $|
E(\pi_3 \cap \pi_4)| = j$ and $\pi_3 \cap \pi_4$ is acyclic with $k + 1$ components, we have that 
$|V(\pi_3 \cap \pi_4)| = j + k + 1$. Now suppose it were shown that $\pi_3 \cup \pi_4$ can be made 
acyclic by removing at most $k$ edges while keeping the vertex set same and call this new graph as 
$H$. One would then have,
\begin{equation*}
|V\big(H \cap (\pi_1 \cup \pi_2)\big)| \geq |E\big(H \cap (\pi_1 \cup \pi_2)\big)| + 1 \geq |
E\big((\pi_3 \cup \pi_4) \cap (\pi_1 \cup \pi_2)\big)| - k + 1 = j' - k + 1\,.
\end{equation*}
Adding this to $|V(\pi_3 \cap \pi_4)| = j + k + 1$ would immediately give \eqref{vertex_count}. In 
the remaining part of this proof we will show that one can remove $k$ edges from $\pi_3 \cup 
\pi_4$ so that the resulting graph becomes acyclic.\par

Let $\mathcal C$ be a cycle in $\pi_3 \cup \pi_4$. Since $\pi_3$ and $\pi_4$ are acyclic, 
$\mathcal C$ consists of an alternating sequence of segments in $E(\pi_4) \setminus E(\pi_3)$ 
and $E(\pi_3) \setminus E(\pi_4)$ interspersed with segments in any one of the $C^O_i$'s (possibly 
trivial i.e.~consisting of a single vertex). This implies that for some $1 \leq i, i' \leq 
k+1$, $\mathcal C$ contains a (nontrivial i.e.~of positive length) segment in $E(\pi_4) \setminus 
E(\pi_3)$ joining $C^O_i$ and $C^O_{i'}$. In fact $i \neq i'$ since $\pi_4$ is acyclic. Hence the 
only case we need to consider is when $k \geq 1$. As $\pi_4$ is a path, $C^O_1, C^O_2, \cdots, 
C^O_{k+1}$'s are vertex-disjoint segments (possibly trivial) aligned along $\pi_4$ in some order 
with $k$ intervening (nontrivial) segments in $E(\pi_4) \setminus E(\pi_3)$. Pick one edge from 
each of these $k$ segments. It follows from the discussions so far that $\mathcal C$ must contain 
one of these edges. Consequently removing these $k$ edges from $\pi_3 \cup \pi_4$ would make the 
resulting graph acyclic. We refer the readers to Figure \ref{fig:fig2} for an illustration.
\end{proof}
We will now use \eqref{vertex_count} and Lemma~\ref{lemma_crucial1} to show that $N^*_\ell$ and 
$N_\ell'$ concentrate around their expected values.
\begin{lemma}
\label{lemma_crucial3}
Assume the same conditions on $\zeta_1, \zeta_2$, $\ell$ and $\eta$ as in 
Lemma~\ref{lemma_crucial1}. Then there exists $g_{\ell, \eta} = g_{\ell, \eta; \zeta_1, \zeta_2}: 
\mathbb N \mapsto [0, \infty)$ depending on $\ell, \eta$ (and $\zeta_1, \zeta_2$) with $g_{\ell, 
\eta}(n)\to 0$ as $n\to \infty$ such that the 
following hold:\\
(1) $\mathbb{P}\big(|N^*_\ell - \mathbb{E} N^*_\ell|  \leq g_{\ell,\eta}(n) 
\mathbb{E}N^*_\ell\big) \to 1$ as $n \to \infty$;\\
(2) $\mathbb{P}\big(|N_\ell' - \mathbb{E} N_\ell'|  \leq g_{\ell,\eta}(n)\mathbb{E} N_\ell'\big) 
\to 1$ as $n \to \infty$.
\end{lemma}
\begin{proof}
The proof of (1) is rather straightforward. By \eqref{sec_mom_cond} and \eqref{sum2} we see that
\begin{eqnarray*}
\mathbb{E}((N^*_\ell)^2) \leq (\mathbb{E}N^*_\ell)^2(1 + o(1))\,.
\end{eqnarray*}
An application of Markov's inequality then yields Part (1). 
In order to prove Part (2), we first argue that $\mathbb{E}N_\ell' = \Theta(n)$. Similar to the 
computation of \eqref{first_bnd}, we can show that $\mathbb{E}N^*_\ell$ is $O(n)$. But then 
\eqref{expected_edge} tell us that same is also true for $\mathbb{E}N_\ell'$. For the lower bound, 
notice that given any path $\pi_1$ in $\Pi_\ell^*$, there are $\Theta(n^{\ell})$ many paths in 
$\Pi_\ell^*$ that intersect $\pi_1$ in exactly one vertex. Furthermore for any such pair $(\pi_1, 
\pi_2)$ we have
$$\mathbb{P}(H_{\pi_1,\pi_2})= (\mathbb{P}(G_{\pi}))^2 = \Theta(n^{-2\ell}),$$ where the last 
equality follows from \eqref{gamma_density} (see the computation in \eqref{first_bnd}) and 
Lemma~\ref{cond_dev1}. Therefore, we obtain that
\begin{eqnarray*}
\mathbb{E}N_\ell' = \Theta(n^{\ell + 1}) \mbox{$\sum_{\pi_2 \in \Pi_{\ell,\pi_1}^*}$}\mathbb{P}
(G_{\pi_1} \cap G_{\pi_2})\geq  \Theta(n^{\ell + 1}) \Theta(n^{\ell}) \Theta(n^{-2\ell})= 
\Theta(n).
\end{eqnarray*}
Next we estimate $\E ((N_\ell')^2)$. For this purpose, we first consider two fixed $\pi_1, 
\pi_2\in \Pi^*_\ell$ such that $V(\pi_1) \cap V(\pi_2) \neq \emptyset$. For $0 \leq j \leq \ell$ 
and $1 \leq j' \leq 2\ell - j$, let $\Pi_{\pi_1, \pi_2}^{\ell,j,j'}$ be the collection of all 
pairs of paths $(\pi_3, \pi_4) \in \Pi_\ell^*$ such that $|E(\pi_1 \cup \pi_2) \cap E(\pi_3 \cup 
\pi_4)| = j'$ and $|E(\pi_3 \cup \pi_4)| = 2\ell - j$.
For $(\pi_3, \pi_4)\in \Pi_{\pi_1, \pi_2}^{\ell,j,j'}$, we see that $|E(\pi_3 \cup \pi_4) 
\setminus E(\pi_1 \cup \pi_2)| = 2\ell - j - j'$ and thus by a similar reasoning as employed in 
\eqref{cond_prob_order} we get
$$\mathbb{P}(H_{\pi_3, \pi_4}|H_{\pi_1, \pi_2}) = O(n^{j + j' - 2\ell})\,.$$ Now let $\Pi_{\pi_1, 
\pi_2}^{\ell,j,j'}(n_1, n_2) \subseteq \Pi_{\pi_1, \pi_2}^{\ell,j,j'}$ contain all the pairs 
$(\pi_3, \pi_4)$ such that $|V(\pi_3) \cap V(\pi_4)| = n_1 \geq 1$ and  $|V(\pi_3 \cup \pi_4) \cap 
V(\pi_1 \cup \pi_2)| = n_2$. Then $|V(\pi_3 \cup \pi_4) \setminus V(\pi_1 \cup \pi_2)| = 2\ell + 2 
- n_1 - n_2$ and consequently $|\Pi_{\pi_1, \pi_2}^{\ell,j,j'}(n_1, n_2)| = O(n^{2\ell + 2 - n_1 - 
n_2})$.
By Lemma~\ref{count_vertices}, we know that for $n_1 + n_2 \geq j + j' +2$ for $(\pi_3, \pi_4)\in 
\Pi_{\pi_1, \pi_2}^{\ell,j,j'}(n_1, n_2)$. Therefore, 
\begin{equation*}
\mbox{$\sum_{(\pi_3, \pi_4) \in \Pi_{\pi_1, \pi_2}^{\ell,j,j'}}$}\mathbb{P}(H_{\pi_3, \pi_4}|
H_{\pi_1, \pi_2}) = \mbox{$\sum_{1\leq n_1, n_2 \leq 2\ell+2} \sum_{(\pi_3, \pi_4) \in \Pi_{\pi_1, 
\pi_2}^{\ell,j,j'}(n_1, n_2)}$} \mathbb{P}(H_{\pi_3, \pi_4}|H_{\pi_1, \pi_2}) = O(1)\,.
\end{equation*}
This implies that
$$ \mbox{$ \sum_{(\pi_1, \pi_2), (\pi_3, \pi_4)}  $}\mathbb P(H_{\pi_1, \pi_2} \cap H_{\pi_3, 
\pi_4}) = O(1) \E  N'_\ell\,.$$
where the sum is over all such pairs such that $ |E(\pi_1 \cup \pi_2) \cap E(\pi_3 \cup \pi_4)| 
\neq \emptyset$. In addition, 
$$ \mbox{$ \sum_{(\pi_1, \pi_2), (\pi_3, \pi_4)}  $}\mathbb P(H_{\pi_1, \pi_2} \cap H_{\pi_3, 
\pi_4}) = (1+o(1))( \E  N'_\ell)^2\,.$$
where the sum is over all such pairs such that $ |E(\pi_1 \cup \pi_2) \cap E(\pi_3 \cup \pi_4)|  =  
\emptyset$ (thus in this case $H_{(\pi_1, \pi_2)}$ is independent of $H_{(\pi_3, \pi_4)}$). 
Combined with the fact that $\E N'_\ell = \Theta(n)$, it gives that $\E ((N'_\ell)^2) = (1+o(1)) 
(\E N'_\ell)^2$. At this point, another application of Markov's inequality completes the proof of 
the lemma.
\end{proof}
We are now well-equipped to prove the main lemma of this subsection. For convenience of notation, 
write 
\begin{equation}\label{eq-def-f-l-eta}
f(\ell, \eta) = f_{\zeta_2}(\ell, \eta) = \mathrm{e}^{-1000\zeta_2/\sqrt{\eta}}\sqrt{\eta^3 / 
\ell^7}\,.
\end{equation}
\begin{lemma}
\label{lemma_crucial2}
Assume the same conditions on $\zeta_1, \zeta_2$, $\ell$ and $\eta$ as in 
Lemma~\ref{lemma_crucial1}. Let $S_{n, \eta, \ell} = S_{n, \eta, \ell; \zeta_1, \zeta_2}$ be a set 
with maximum cardinality among all subsets of $\Pi_\ell^*$ containing only pairwise disjoint good 
paths. Then there exists an absolute constant  $C_6>0$ such that, 
\begin{equation}
\label{num_dis_good}
\mathbb{P}(|S_{n, \eta, \ell}| \geq C_6f(\ell, \eta) n) \to 1\mbox{ as }n\to\infty\,.
\end{equation}
\end{lemma}
\begin{proof}
Let $h(\ell, \eta) = C_5\mathrm{e}^{1000\zeta_2/\sqrt{\eta}}\sqrt{\ell^7 / \eta^3}$. By 
Lemma~\ref{lemma_crucial3} and \eqref{expected_edge}, we assume without loss of generality that 
$$|N^*_\ell - \mathbb{E}N^*_{\ell}| \leq g_{\ell,\eta}(n)\mathbb{E}N^*_{\ell} \mbox{ and }
  N_\ell' \leq  (1 + o(1))h(\ell, \eta)\tfrac{(\mathbb{E}N^*_\ell)^2}{n} (1 + g_{\ell,\eta}(n)),$$
where $g_{\ell, \eta}(n)$ is defined as in Lemma~\ref{lemma_crucial3}. Since $N_\ell' = 2|
E(\mathcal{G}_n)| + |V(\mathcal{G}_n)|$, by Lemma~\ref{relation} we get that the graph 
$\mathcal{G}_n$ has an independent subset of size at least
$$ {N^*_\ell}^2 / N_{\ell}' \geq n(1+o(1))/h(\ell,\eta).$$
Therefore, with high probability $|S_{n, \eta, \ell}| \geq n/2h(\ell,\eta)$ which leads to 
\eqref{num_dis_good} for $C_6= 1/2C_5$.
\end{proof}

\subsection{Connecting short light paths into a long one}
We set $\zeta_1 = 1/5$ and $\zeta_2 = 1 + \sqrt{2C^*/\mathrm{e}}$ in this subsection. Note that 
this choice satisfies the conditions in Lemma~\ref{lemma_crucial2}. Denote by $\mathcal{E}_{n, 
\eta, \ell}$ the event $\{|S_{n, \eta, \ell}| \geq C_6f(\ell, \eta) n)\}$.\\
The remaining part of our scheme is to connect a fraction of these disjoint good paths in a 
suitable way to form a light and long path $\gamma$. In order to describe our algorithm for the 
construction of $\gamma$, we need a few more notations. Denote the vertex sets 
$V(\mathcal{W}_n^*)$ and $V(\mathcal{W}_n)\setminus V(\mathcal{W}_n^*)$ by $V_1$ and $V_2$ 
respectively. Let $\delta > 0$ be a number and $\nu > 0$ be an integer satisfying 
\begin{equation}\label{algo_inequation}
1 \leq \delta n /\ell \leq |S_{n, \eta, \ell}|\mbox{ and }\delta n \nu / \ell \leq |V_2|.
\end{equation}
Now label the paths in $S_{n, \eta, \ell}$ as $\pi_1, \pi_2, \ldots$ in some arbitrary way. Our 
aim is to build up the path $\gamma$ in step-by-step fashion starting from $\pi_1$. In each step 
we will connect $\gamma$ to some $\pi_j$ by a path of length 2 whose \emph{middle} vertex is in 
$V_2$. These paths will be referred to as \emph{bridges}. To leverage additional flexibility we 
also demarcate two segments of length $\lfloor \ell/4 \rfloor$ one on each end of the paths 
$\pi_j$'s which we call \emph{end segments}. 
These end segments will allow us to ``choose'' endpoints of $\pi_j$'s while connecting them (as 
such, it is possible that we only keep half of the vertices of $\pi_j$ in $\gamma$). A vertex $v$ 
will be said to be adjacent to a path or an edge if it is an endpoint of that path or edge. If an 
edge $e$ has exactly one endpoint in $S$, we denote that endpoint by $v_{e,S}$. The following 
algorithm, referred to as $\mathrm{BRIDGE}(\nu, \ell, \delta)$, will construct a long path 
$\gamma$. 
See Figure \ref{fig:fig1} for an illustration.
\smallskip

\textbf{Initialization.} $\gamma = \pi_1$, $T$ is the set of all vertices which are in end 
segments of $\pi_j$'s for $j \geq 2$, $M = V_2$, $P = \emptyset$ and designate an end segment 
of $\gamma$ as the open end $\gamma_O$. Also let $v$ be the endpoint of $\gamma$ \textbf{not} in 
$\gamma_O$.

Now repeat the following sequence of steps $\lfloor \delta n / \ell \rfloor - 1$ times:

\textbf{Step 1.} Repeat $\nu$ times: find the lightest edge $e$ between $\gamma_O$ and $M$, remove 
$v_{e,M}$ from $M$ and include it in $P$. These $\nu$ edges will be called predecessor edges 
(so at the end of this step, $|P| = \nu$).

\vspace{0.05 cm}

\textbf{Step 2.} Find the lightest edge between $P$ and $T$. Call it $e'$. Then $v_{e',T}$ 
comes from an end segment of some path in $S_{n, \eta, \ell}$, say $\pi$.

\vspace{0.05 cm}

\textbf{Step 3.} The edge $e'$ and the \emph{unique} predecessor edge adjacent to $v_{e',P}$ 
defines a path $b$ of length 2 (so $b$ connects a vertex in $\gamma_O$ to a vertex in $\pi$). Let 
$w$ be the endpoint of $\pi$ not in the end segment that $v_{e',T}$ came from. Then there is a 
unique path $\gamma'$ in the tree $\gamma \cup b \cup \pi$ between $v$ and $w$. Set $\gamma = 
\gamma'$ and $\gamma_O =$ the end segment of $\pi$ containing $w$.

\vspace{0.05 cm}

\textbf{Step 4.} Remove the vertices on the end segments of $\pi$ from $T$ and reset $P$ at 
$\emptyset$.\newline 
\begin{figure}[!ht]
 \centering\includegraphics[width=0.8\textwidth]{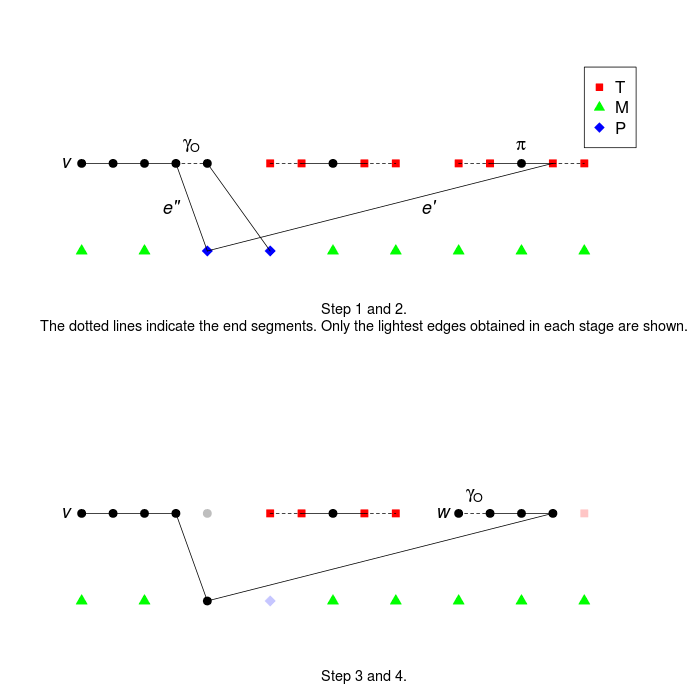}
 \caption{{\bf Illustrating an iteration of BRIDGE for $\nu = 2$ and $\ell = 4$.} The edges $e'$ 
 and $e''$ define the path $b$. So in this iteration the paths $\gamma$ and $\pi$ are shortened 
 slightly before being joined via $b$.}
 \label{fig:fig1}
\end{figure}
\newline
Notice that the conditions in \eqref{algo_inequation} ensure that we never run out of vertices in 
$T$ or $M$ during first $\lfloor \delta n / \ell \rfloor - 1$ iterations of steps 1 to 4. Thus 
what we described above is a valid algorithm for such choices of $\delta$ and $\nu$. Denote the 
length and average weight of the path $\gamma$ generated by $\mathrm{BRIDGE}(\nu, \ell, \delta)$ 
as $L_{\text{BRIDGE}}(\nu, \ell, \delta)$ and $A_{\text{BRIDGE}}(\nu, \ell, \delta)$ respectively 
when $\delta, \nu, \ell$ satisfy these inequalities. For sake of completeness we may define these 
quantities to be $0$ and $\infty$ respectively and regard the output path $\gamma$ as ``empty'' if 
any one of the inequalities in \eqref{algo_inequation} fails to hold. We are now just one lemma 
short of proving the lower bound in \eqref{main_prop}.
\begin{lemma}
\label{algo}
For any $0 < \eta < \eta_2$ where $\eta_2 > 0$ is an absolute constant there exist positive 
integers $\nu = \nu(\eta)$, $\ell = \ell(\eta) \geq \zeta_2^2 / \eta$ and a positive number 
$\delta = \delta(\eta)$ such that
$$\mathbb{P}\big(L_{\text{BRIDGE}}(\nu, \ell, \delta) \geq \mathrm{e}^{-C_{7}/\sqrt{\eta}} n\mbox{ 
and }A_{\text{BRIDGE}}(\nu, \ell, \delta) \leq 1/\mathrm{e} + 12\eta \mid \mathcal{E}_{n, \eta, 
\ell}\big) \to 1$$
as $n$ tends to infinity. Here $C_{7}>0$ is an absolute constant.
\end{lemma}
\begin{proof}
We will omit the phrase ``conditioned on $\mathcal{E}_{n, \eta, \ell}$'' while talking about 
probabilities in this proof (barring formal expressions) although that is to be implicitly 
assumed. We use $\mathrm{Exp}(1 / \theta)$ to denote the distribution of an exponential random 
variable with mean $\theta > 0$. Define $\mathcal B_{n, \eta, \nu, \ell, \delta}$ to be the 
event that the total weight of bridges does not exceed $3\ell\eta \times \lfloor \delta n/\ell 
\rfloor$. Notice that if any one of the inequalities in \eqref{algo_inequation} does not hold, 
$\gamma$ is ``empty'' and hence $\mathcal B_{n, \eta, \nu, \ell, \delta}$ is a sure event. Suppose 
$\delta, \nu$ and $\ell$ are such that \eqref{algo_inequation} is satisfied.  We will first bound 
the average weight $A(\gamma)$ of $\gamma$ assuming that $\mathcal B_{n, \eta, \nu, \ell, \delta}$ 
occurs. Let $\ell_i$ be the length of the segment selected by the algorithm in the $i$-th 
iteration. We see that its weight can be no more than $\lambda \ell_i + 2\zeta_2 / \sqrt{\eta}$, 
since the segment is chosen from a good path of average weight at most $\lambda$ and maximum 
deviation from its linear interpolation is at most $\zeta_2 / \sqrt{\eta}$ (see \eqref{good_event} 
as well as the proof for Lemma~\ref{conditional_prob}). Thus the total weight of edges in $\gamma$ 
from the good paths is bounded by $\lambda L + \lfloor \delta n / \ell \rfloor.(2\zeta_2 / 
\sqrt{\eta})$ where $L = \sum_{i}\ell_i$. Adding this to the total weight of bridges we get with 
probability tending to 1 as $n\to \infty$
\begin{equation*}
W(\gamma) \leq L( 1 / \mathrm{e} + \eta) + \lfloor \delta n / \ell \rfloor \cdot (2\zeta_2 / 
\sqrt{\eta}) + \lfloor \delta n/\ell \rfloor3\ell\eta.
\end{equation*}
Since the algorithm selects at least $\ell / 2$ edges from each of the $\lfloor \delta n / \ell 
\rfloor$ good paths it connects, we have $\ell_i\geq \ell/2$ for each $i$ and thus $L \geq \lfloor 
\delta n / \ell \rfloor \times \ell/2$. Therefore, 
\begin{eqnarray*}
A(\gamma) 
       &\leq & 1/\mathrm{e} + \eta + \lfloor \delta n / \ell \rfloor.(2\zeta_2 / L\sqrt{\eta}) + 
       \lfloor \delta n/\ell \rfloor3\ell\eta / L       
       \\
       &\leq & 1 / \mathrm{e} + \eta + 4\zeta_2 / \ell\sqrt{\eta} + 6\eta\,.
\end{eqnarray*}
If $\ell\geq \zeta_2 / \eta^{3/2}$, then from the last display we can conclude $A(\gamma) \leq 1 / 
\mathrm{e} + 12\eta$. We can assume this restriction on $\ell$ for now. Indeed, later we will 
specify the value of $\ell$ and it will satisfy the condition $\ell\geq \zeta_2 / \eta^{3/2}$.\par

So it remains to find positive numbers $\delta, \nu, \ell$ as functions of $\eta$ and an absolute 
constant $\eta_2 > 0$ such that the following three hold for all $0 < \eta < \eta_2$: (a) 
$\mathbb{P}(\mathcal B_{n, \eta, \nu, \ell, \delta} \mid \mathcal E_{n, \eta, \ell}) \rightarrow 
1$ as $n \to \infty$, (b) $\ell \geq \zeta_2 / \eta^{3/2} \vee \zeta_2^2 / \eta$ (see the 
statement of the lemma as well as the last paragraph) and (c) $\gamma$ has the desired length. In 
the next paragraph we will find a triplet $(\delta, \nu, \ell)$ and an absolute constant $\eta_2'' 
> 0$ such that (a) holds for $0 < \eta < \eta_2''$. In the final paragraph we will show that our 
choice of $(\delta, \nu, \ell)$ also satisfies (b) and (c) whenever $0 < \eta < \eta_2$ where 
$\eta_2 < \eta_2''$ is an absolute constant.\par

Let us begin with the crucial observation that, at the start of each iteration the edges between 
$M$ and $\gamma_O$ are still unexplored. The same is true for the edges between $P$ and $T$ at the 
end of Step 1 in any iteration. Consequently their weights are i.i.d.\ $\mathrm{Exp}(1/n)$ 
regardless of the outcomes from the previous iterations. Therefore, all the bridge weights are 
independent of each other. Now suppose the mean and variance of each bridge weight can be bounded 
above by $2\ell\eta$ and $\sigma^2$ respectively and we emphasize that the latter does not depend 
on $n$. By Markov's inequality it would then follow that $\lim_{n \to \infty}\mathbb{P}(\mathcal 
B_{n, \eta, \nu, \ell, \delta} \mid \mathcal E_{n, \eta, \ell}) = 1$. To that end let us consider 
the bridge obtained from the $m$-th iteration where $1\leq m\leq \lfloor \delta n/\ell \rfloor - 
1$. Note that here we implicitly assume \eqref{algo_inequation}, but this would be shortly shown 
to be implied by some other constraints involving $\delta, \nu$ and $\ell$. Let $e'$ be the 
lightest edge between $P$, $T$ in Step 2 and $e$ be the predecessor edge adjacent to $e'$ (for 
this iteration). So the bridge weight is simply $W_{e'} + W_e$. By discussions on independence at 
the beginning of the proof, it follows that $W_{e'}$ and $W_e$ are independent of each other and 
also of the weights of bridges already chosen. Since these weights are minima of some collections 
of i.i.d.\ exponentials, they will be of small magnitude provided that we are minimizing over a 
large collection of exponentials, i.e., $|T|$, $|M|$ and $\nu$ are big. It follows from the 
description of the algorithm that at each iteration we lose $2\lfloor \ell/4\rfloor$ many vertices 
from $T$ and $\nu$ many vertices from $M$. By simple arithmetic we then get,
\begin{equation}\label{eq-size-P1P2}
|T| \geq C_6\lfloor\ell /4\rfloor f(\ell,\eta)n \mbox{ and }|M| \geq \zeta_1 \eta n/ 2\,,
\end{equation}
for all $1 \leq m \leq \lfloor \delta n / \ell\rfloor - 1$ provided
\begin{equation}
\label{eq-delta-zeta-0}
\delta \leq C_6 f(\ell,\eta)\ell / 2 \mbox{ and } \nu \delta / \ell \leq \zeta_1 
\eta / 2\,.
\end{equation}
Notice that these inequalities automatically imply $\delta n / \ell \leq C_6 f(\ell, \eta)n$ and 
$\delta n \nu / \ell \leq |V_2|$. Thus if $\delta, \nu, \ell$ satisfy \eqref{mean_bnd2}, 
\eqref{algo_inequation} would also be satisfied for all large $n$ (given $\delta, \ell$). Assume 
for now that \eqref{eq-delta-zeta-0} holds. Since $W_{e'}$ is minimum of $\nu \times |T|$ many 
independent $\mathrm{Exp}(1/n)$ random variables, it is distributed as $\mathrm{Exp}(\nu|T| / n)$. 
As for $W_{e}$, it is bounded by the maximum weight of the $\nu$ predecessor edges. From 
properties of exponential distributions and description of the algorithm it is not difficult to 
see that this maximum weight is distributed as $E_1 + E_2 + \ldots E_\nu$, where $E_{i+1}$ is 
exponential with rate $(|M| - i)\times 1 / n\times \lfloor \ell / 4\rfloor$. By \eqref{eq-size-P1P2}, we can then bound the expected weight of the bridge from above by
\begin{eqnarray}\label{mean_bnd1}
\tfrac{1}{C_6\lfloor\ell /4\rfloor f(\ell,\eta)n} \times \tfrac{1}{\nu} \times n\mbox{ }+\mbox{ 
}\tfrac{\nu}{\big(\frac{\zeta_1\eta}{2} - \tfrac{\nu}{n}\big)\lfloor\ell /4\rfloor} 
\leq  \tfrac{5}{C_6\nu\ell f(\ell,\eta)} + \tfrac{11\nu}{\zeta_1\eta \ell},
\end{eqnarray}
where the last inequality holds for $\ell \geq 20$ and large $n$ (given $\eta$, $\nu$). 
By the same line of arguments, we get that the its variance is bounded by a number that depends 
only on $\eta$, $\ell$ and $\nu$ (so in particular independent of $n$). To make the right hand 
side of \eqref{mean_bnd1} bounded above by $2\ell\eta$, we may require each of the summands in 
\eqref{mean_bnd1} to be bounded by $\ell\eta$. After a little simplification this amounts to
\begin{equation}
\nu \geq 5 / C_6\ell^2\eta f(\ell, \eta) \,, \mbox{ and } \zeta_1 (\ell\eta)^2 \geq 11 \nu\,.
\label{mean_bnd2}
\end{equation}
So we need to pick a positive $\delta = \delta(\eta)$, positive integers $\nu = \nu(\eta), \ell = 
\ell(\eta)$ and an absolute constant $\eta_2'' > 0$ such that \eqref{eq-delta-zeta-0} and 
\eqref{mean_bnd2} hold for $0 < \eta < \eta_2''$. We will deal with \eqref{mean_bnd2} first which 
is in fact equivalent to
\begin{equation}
\label{ineq_solve_1}
\zeta_1(\ell \eta)^2 / 11 \geq \nu \geq 5 / C_6\ell^2f(\ell, \eta)\,.
\end{equation}
Let us try to find an integer $\ell$ satsfying $\zeta_1(\ell \eta)^2 / 11 \geq \big(10 / 
C_6\ell^2f(\ell, \eta)\big)\vee 2$ since this will ensure the existence of a positive integer 
$\nu$ such that $\nu, \ell$ satisfy \eqref{ineq_solve_1}. Using $f(\ell, \eta) = 
\mathrm{e}^{-1000\zeta_2/ \sqrt{\eta}}\sqrt{\eta^3 / \ell^7}$, we get that this amounts to
$$\ell \geq \tfrac{C_7'\mathrm{e}^{2000\zeta_2 / \sqrt{\eta}}}{\eta^9} \vee \tfrac{C_7''}{\eta}\,,
$$
for some positive, absolute constants $C_7'$ and $C_7''$. Hence there exists an absolute constant 
$\eta_2''' > 0$ such that the integers $\ell = \lceil \mathrm{e}^{2001\zeta_2 / \sqrt{\eta}} 
\rceil$ and $\nu = \lfloor \zeta_1(\ell \eta)^2/ 11 \rfloor$ satisfy \eqref{ineq_solve_1} whenever 
$0 < \eta < \eta_2'''$. Now we need to find $\delta$ that would satisfy \eqref{eq-delta-zeta-0} 
which can be rewritten as,

\begin{equation}
\label{ineq_solve_2}
\delta \leq (C_6 f(\ell, \eta)\ell / 2) \wedge (\zeta_1\eta\ell/ 2\nu)\,.
\end{equation}
Again substituting $f(\ell, \eta) = \mathrm{e}^{-1000\zeta_2/ \sqrt{\eta}}\sqrt{\eta^3 / \ell^7}$, 
we can simplify \eqref{ineq_solve_2} to
\begin{equation}
\label{ineq_solve_3}
\delta \leq (C_6 \mathrm{e}^{-1000\zeta_2/\sqrt{\eta}}\tfrac{\eta^{3/2}}{2\ell^{5/2}}) \wedge 
(\zeta_1\eta\ell/ 2\nu)\,.
\end{equation}
Since $\nu = \lfloor \zeta_1(\ell \eta)^2/ 11 \rfloor$, \eqref{ineq_solve_3} would be satisfied if
$$\delta \leq (C_6 \mathrm{e}^{-1000\zeta_2/\sqrt{\eta}}\tfrac{\eta^{3/2}}{2\ell^{5/2}}) \wedge 
(11 / 2\ell \eta)\,.$$
The last display together with our particular choice of $\ell$ i.e. $\lceil 
\mathrm{e}^{2001\zeta_2 / \sqrt{\eta}} \rceil$ imply that there exists a positive, absolute 
constant $\eta_2'' < \eta_2'''$ such that $\delta = \mathrm{e}^{-7000\zeta_2 / \sqrt{\eta}}$ 
satisfies \eqref{ineq_solve_2} for $0 < \eta < \eta_2''$. Thus our choice of the triplet $(\delta, 
\nu, \ell)$ satisfies \eqref{eq-delta-zeta-0} and \eqref{mean_bnd2} for $0 < \eta < \eta_2''$ and 
consequently the event $\mathcal B_{n, \eta, \nu, \ell, \delta}$ occurs with high probability for 
this choice.\par

As to the constraint on $\ell$, it is also clear that there exists a positive, absolute constant 
$\eta_2' < \eta_2''$ such that $\ell = \lceil \mathrm{e}^{2001\zeta_2 / \sqrt{\eta}} \rceil$ is 
larger than $\zeta_2 / \eta^{3/2} \vee \zeta_2^2 / \eta$ for all $0 < \eta < \eta_2'$. Finally it 
is left to ensure whether $\gamma$ has the length required by the lemma. Since our particular 
choice of the triplet $(\delta, \nu, \ell)$ satisfies \eqref{algo_inequation} for large $n$ (given 
$\eta$), we have that $L_{\text{BRIDGE}}(\nu, \ell, \delta) \geq \lfloor \delta n / \ell \rfloor 
\times \ell/2$. It then follows that there exists a positive, absolute constant $\eta_2 < \eta_2'$ 
such that $L_{\text{BRIDGE}}(\nu, \ell, \delta) \geq \mathrm{e}^{-7001\zeta_2/ \sqrt{\eta}}n$ for 
these particular choices of $\nu, \ell$ and $\delta$ whenever $0 < \eta < \eta_2$ and $n$ is large 
(given $\eta$). This completes the proof of the lemma.

\end{proof}

Combining Lemmas \ref{lemma_crucial2} and \ref{algo} completes the proof of the lower bound in 
Theorem~\ref{Prop}.

\end{document}